\documentclass[12pt,a4paper]{amsart}
\theoremstyle{plain}
\usepackage{enumerate,amsfonts,amssymb,mathrsfs,amscd}
\usepackage{amsmath}
\usepackage[english]{babel}
\usepackage[utf8]{inputenc}
\usepackage{amsmath, amsthm, amssymb}
\usepackage[pdftex]{graphicx,color}

\usepackage{tikz,pgf}
\usetikzlibrary{calc}
\usepackage[colorlinks=true,linkcolor=red,citecolor=red]{hyperref}

\usepackage{etex}

\usepackage[shortlabels]{enumitem}
\usepackage{setspace}
\usepackage{tabularx,multirow}
\usepackage[arrow, matrix, curve]{xy}
\usepackage{MnSymbol}

\advance\hoffset-20mm \advance\textwidth40mm

\usepackage[normalem]{ulem}

\theoremstyle{plain}
\newtheorem{theorem}{Theorem}[section]
\newtheorem{corollary}[theorem]{Corollary}
\newtheorem{lemma}[theorem]{Lemma}
\newtheorem{proposition}[theorem]{Proposition}

\theoremstyle{definition}
\newtheorem{definition}[theorem]{Definition}

\newtheorem{remark}[theorem]{Remark}
\newtheorem{remarks}[theorem]{Remarks}
\newtheorem{example}[theorem]{Example}

\newcounter{caseinproof}
0

\def\span{\mathop{\span}}

\def\deg{\mathop{\rm deg}}

\def\exp{\mathop{\rm exp}}

\def\span{\mathop{\rm span}}

\def\ll1{l_{\lambda}^{-1}(1)}
\def\lm1{l_{\mu}^{-1}(1)}

\newcommand\qmatrix[2][1]{\left(\renewcommand\arraystretch{#1}
\begin{equation}gin{array}{*{20}r}#2\end{array}\right)}

\begin{document}

\title{Tits' type alternative for groups acting on toric affine varieties}

\author{Ivan Arzhantsev and Mikhail Zaidenberg}

\address{National Research University Higher School of Economics, Faculty of Computer Science,
Pokrovsky Boulevard 11, Moscow, 109028 Russia}
\email{arjantsev@hse.ru}
\address{Universit\'e Grenoble Alpes, CNRS, Institut Fourier, F-38000 Grenoble, France}
\email{Mikhail.Zaidenberg@univ-grenoble-alpes.fr}

\begin{abstract} Given a toric affine algebraic variety $X$ and a collection of one-parameter unipotent subgroups $U_1,\ldots,U_s$ of $\mathop{\rm Aut}(X)$ which are normalized by the torus acting on $X$, we show that the group $G$ generated by $U_1,\ldots,U_s$ verifies the following alternative of Tits' type:  either $G$ is a unipotent algebraic group, or it contains a non-abelian free subgroup. 
We deduce that if $G$ is $2$-transitive on a $G$-orbit in $X$, then $G$ contains a non-abelian free subgroup, and so, is of exponential growth. 
\end{abstract} 

\maketitle

\section{Introduction} \label{sec:intro}
We fix an algebraically closed field ${\bf k}$ of  characteristic zero.
Let ${\mathbb A}^n$ stand for the affine space of dimension $n$ over ${\bf k}$ 
and $\mathbb{G}_a$ ($\mathbb{G}_m$) for the additive (multiplicative, respectively) group of 
${\bf k}$ viewed as an algebraic group. 
Consider an algebraic variety $X$ over ${\bf k}$ and  an effective regular action $\mathbb{G}_a\times X\to X$. The image of $\mathbb{G}_a$ in $\mathop{\rm Aut} (X)$ is called a \emph{one-parameter unipotent subgroup} of $\mathop{\rm Aut} (X)$, or a \emph{$\mathbb{G}_a$-subgroup}, 
for short.  Any $\mathbb{G}_a$-subgroup $U$ of $\mathop{\rm Aut}(X)$ has the form $U=\{\exp(t\partial)\,|\,t\in{\bf k}\}$, where $\partial$ 
is a locally nilpotent derivation of the structure ring $\mathcal{O}(X)$. This correspondence between the $\mathbb{G}_a$-subgroups and locally nilpotent derivations does not hold in prime characteristic, and so, we prefer in this paper to work in characteristic zero.

The main result of the paper is the following

\begin{theorem}\label{thm:0.4} Consider a toric affine variety $X$ with no torus factor. Let  a subgroup $G$  of 
$\mathop{\rm Aut} (X)$ be generated by a finite collection $U_1,\ldots,U_s$ of one-parameter unipotent subgroups normalized by the acting torus.   
Then  either
\begin{itemize}\item[(i)] $G$ is a unipotent algebraic group, 

\noindent or

\item[(ii)] $G$
contains the free group $\mathbf{F}_2$ of rank two as a subgroup.
\end{itemize} 
\end{theorem}

One says that a variety $X$ over $\bf k$ has a torus factor if $X\cong Y\times (\bf{k}\setminus\{0\})$ for some variety $Y$. 
A toric affine variety $X$ has a torus factor if and only if there is a nonconstant invertible regular function on $X$. 

From Theorem~\ref{thm:0.4} we deduce the following corollary. 

\begin{corollary} \label{thm:0.3}
Let $G$ be a group acting on a toric affine variety $X$ and generated by a finite collection $U_1,\ldots, U_s$ of one-parameter unipotent subgroups normalized by the acting torus. If $G$ 
is doubly transitive on a $G$-orbit in $X$, then $G$ contains a free subgroup of rank two. 
\end{corollary}  
 
 \smallskip
 
The expression ``Tits' type alternative''  in the present paper addresses a property of a class of groups which asserts
that any group from this class either is virtually solvable (resp., virtually nilpotent, virtually abelian, etc.), or contains a non-abelian 
free subgroup. This (rather weak) form of the original Tits alternative disregards whether or not the alternative remains true when passing to a subgroup. We wonder whether, under the setup of Theorem~\ref{thm:0.4}, any (finitely generated) subgroup $H$ of $G$ either is virtually solvable, or contains a free subgroup of rank two. 

Let us provide a brief survey on the classical Tits alternative for the automorphism groups arising in algebraic geometry. Recall that, due to the original Tits' theorem~\cite{Tit:1972},  any finitely generated subgroup of a linear algebraic group  either is virtually solvable, or contains a non-abelian 
free subgroup. Over a field of characteristic zero, the Tits alternative holds for any, not necessarily finitely generated, linear group~\cite{Tit:1972}; in the sequel we call  this property the \emph{enhanced Tits alternative}.
 As an example, we can cite the following result  due to  S.~Cantat and Ch.~Urech.

\begin{theorem}[\cite{Can:2011, Ur:2017}]
The group of birational transformations of  a compact complex K\"{a}hler surface verifies the enhanced Tits alternative, that is, any of its non-virtually solvable subgroups contains a free subgroup of rank two.
\end{theorem}  

This theorem extends the earlier result of S.~Lamy \cite{Lam:2001} which says that $\mathop{\rm Aut} (\mathbb{A}^2)$ verifies the enhanced Tits alternative. 
The Tits alternative holds also for  the tame automorphism group  of $\mathop{\rm SL}_2(\mathbb{C})$ viewed as an affine  quadric threefold 
 \cite[Thm.~C]{BFL:2014}. 
The enhanced Tits alternative is known to hold for the automorphism group of any compact K\"{a}hler variety \cite[Thm.~1.1]{DHZ:2015} (cf.\ also~\cite{Hu:2019} for the positive characteristic case), and as well for the group of birational transformations of any hyperk\"{a}hler variety \cite{Og:2006}; see also~\cite{KY:2019}.

Our starting point was actually the transitivity issue, see Corollary~\ref{thm:0.3}. 
Let $X$ be a toric affine variety over ${\bf k}$ of dimension at least two with no torus factor, and let $\mathop{\rm SAut} (X)\subset\mathop{\rm Aut} (X)$ 
be the subgroup  generated by all the $\mathbb{G}_a$-subgroups of $\mathop{\rm Aut}  (X)$. It is known \cite[Thm.\ 2.1]{AKZ:2012} 
that ${\rm SAut}(X)$ acts highly transitively~\footnote{Or  infinitely transitively, in the terminology of \cite{AFK$^+$:2013}.} on the smooth locus  ${\rm reg} (X)$, that is, 
$m$-transitively for any $m\ge 1$. 
A variety $X$ satisfying the latter property is called \emph{flexible}; 
see \cite[Thm.~1.1]{AFK$^+$:2013} for a criterion of flexibility. 
Notice that an algebraic subgroup 
$G\subset\mathop{\rm Aut} (X)$ cannot act highly transitively on its orbit, by a dimension count argument. 

A $\mathbb{G}_a$-subgroup acting on a toric variety $X$ is called a \emph{root subgroup} if it is normalized by the acting torus. 
The term \emph{root subgroup} is due to the fact that any such subgroup is associated with a certain lattice vector called a \emph{Demazure root},  see subsection \ref{def:roots}.
Assuming in addition that a toric affine variety $X$ is smooth in codimension two, 
one can find a finite number of  root subgroups $U_1,\ldots,U_s$ of $\mathop{\rm Aut}  (X)$ such that the group $G=\langle U_1,\ldots,U_s\rangle$ 
generated by these subgroups still acts highly transitively on ${\rm reg} (X)$ \cite[Thm.~1.1]{AKZ:2019}. \footnote{It is conjectured \cite[Conj.~1.1]{AKZ:2019} 
that an analogous result holds for any flexible affine variety.} 
If $X={\mathbb A}^n$, $n\ge 2$, then just three  $\mathbb{G}_a$-subgroups (which are not root subgroups, in general) are enough 
\cite[Thm.~1.3]{AKZ:2019}; such subgroups are found explicitly in \cite{And:2019}. For instance, 
for $n=2$ the group $G$ generated by the root subgroups 
\[U_1=\{(x,y)\mapsto (x+t_1y^2,y)\}\quad\mbox{and}\quad U_{2}=\{(x,y)\mapsto (x,y+t_2x)\},\quad t_1,\,t_2\in{\bf k}\] 
acts highly  transitively on $\mathbb{A}^2\setminus \{0\}$ equipped with the standard action of the 2-torus, see \cite[Cor.~21]{LPS:2018}. 
Adding one more root subgroup \[U_3=\{(x,y)\mapsto (x+t_3,y)\},\quad t_3\in{\bf k},\] one gets the group 
$\langle U_1, U_2, U_3\rangle$ acting highly transitively on $\mathbb{A}^2$ (cf.\ \cite{Chis:2018}).

The following question arises: \emph{What can one say about a group acting highly transitively 
on an algebraic variety?} More specifically, let us formulate the following conjecture. 

\smallskip

\noindent{\bf Conjecture 1.}\label{conj:general} \emph{Let $X$ be an affine variety over ${\bf k}$ of dimension $\ge 2$. 
Consider the group \[G=\langle U_1,...U_s\rangle\] generated by $\mathbb{G}_a$-subgroups $U_1,...U_s$  of $\mathop{\rm Aut} (X)$. 
Suppose $G$ is doubly transitive on a $G$-orbit. Then $G$ contains a non-abelian free subgroup.}

\smallskip

Corollary~\ref{thm:0.3} partially  confirms  Conjecture~\ref{conj:general}.
Of course, an analog of this conjecture makes sense in different categories. For instance, one might ask (following a referee's suggestion) whether any highly transitive group of homeomorphisms of a compact manifold contains a non-abelian free subgroup; see, e.g., \cite{Wit:1967} for examples of   highly transitive groups of homeomorphisms.  However, the group-combinatorial analog of the conjecture fails; indeed, the  torsion group of finite permutations of $\mathbb{N}$ is highly transitive. The same holds for the infinite alternating group, that is, the simple group of finite even permutations of $\mathbb{N}$.

 Conjecture 1 is inspired in turn by the following question proposed by J.-P.~Demailly: 

\smallskip

\noindent{\bf Question 1.}\label{ques:1} \emph{What can one say about the growth of the group \[G=\langle U_1,...,U_s\rangle\] generated by a sequence of one-parameter unipotent subgroups,
meaning by ``growth'' the maximal growth of the finitely generated subgroups of $G$?}

\smallskip

For instance, the group $G$ in Conjecture  \ref{conj:general} 
has exponential growth provided the conjecture is true. 
Anyway, this group $G$ cannot have a polynomial growth, see Proposition~\ref{lem:subnormal}. The group $G$ in Theorem~\ref{thm:0.4} is of polynomial growth in case (i), and of exponential  growth in case (ii); the latter holds as well for the group $G$ in Corollary~\ref{thm:0.3}.
In the combinatorial setup, we do not know the answer to the following general question.

\smallskip

\noindent {\bf Question 2.} \label{ques:2} \emph{Let $G$ be a finitely generated group. Assume $G$ acts highly transitively 
on a set $X$. Can $G$ be of intermediate growth?}

\smallskip

\noindent See, e.g., \cite{FMS:2015, FLMMS:2020, GG:2013, HO:2016} for recent studies on highly transitive actions of countable groups, and \cite{HO:2016, FLMMS:2020} for surveys. However, the groups of algebro-geometric nature that we study in this paper are quite different. 

\smallskip

The proof of our main Theorem~\ref{thm:0.4} exploits a constructive criterion/algorithm to decide whether the group $G$ in this theorem is a  unipotent algebraic group. We introduce a combinatorial data associated to  the given collection of the one-parameter unipotent subgroups $U_1,\ldots,U_s$ acting on our toric variety $X$. This data is expressed in terms of Demazure roots  $(\rho, e)$. To a Demazure root there corresponds a root derivation $\partial_{\rho,e}$ acting on the structure ring $\mathcal{O}(X)$. It can be viewed as a vector field,  and it generates a root subgroup $U_{\rho,e}$. If $G$ does not contain any non-abelian free group, then there are strong constraints on the Lie brackets between the root derivations generating the root subgroups of $G$; namely, the bracket of any two such derivations is proportional to one of them. These constraints are encoded in a directed graph $\Gamma$ whose vertices are certain abelian Lie algebras which are indexed via the facets of the associated polyhedral cone of $X$ and generated by the corresponding root derivations; see Definition~\ref{def:graph}. Any edge of $\Gamma$ is oriented in the direction of the bracket of its end vertices provided the corresponding subalgebras do not commute; otherwise, the edge is absent. 
The geometry of $\Gamma$ determines the structure of $G$. It occurs that $\Gamma$ has no oriented cycle if and only if it has no bidirected edge, if and only if $G$ is a unipotent algebraic group, see Proposition~\ref{prop:2-cycle}.  
 Theorem~\ref{thm:0.4} is a byproduct of this criterion. 

The content of this paper is as follows. Besides the Introduction, the paper includes four sections. Section~\ref{sec:Preliminaries} contains the notation and preliminary facts from toric geometry. In Section~\ref{sec:2-sbgrps} we prove Theorem~\ref{thm:0.4} in the particular case of a group $G$ generated by  just  two root subgroups, see Proposition~\ref{prop:1.3}.  The main results of subsections~\ref{ss:anla} and~\ref{ss:laug} are Propositions~\ref{prop:ALS}  and~\ref{prop:2-cycle}, respectively. The former contains a combinatorial criterion for a Lie algebra of derivations to be nilpotent and finite dimensional. The latter provides, in our framework, a link between nilpotent Lie algebras and unipotent algebraic groups. Together, these give Proposition~\ref{thm:2} which says that if any two root subgroups from the group $G=\langle U_1,\ldots,U_s\rangle$ generate a unipotent algebraic group, then $G$ itself is a unipotent algebraic group. Theorem~\ref{thm:0.4} follows immediately from Propositions~\ref{prop:1.3} and \ref{thm:2}. Corollary~\ref{thm:0.3} follows from this theorem due to Proposition~\ref{lem:unipotent-transitivity} in subsection~\ref{ss:2-transitive}. According to this proposition, a unipotent linear algebraic group cannot act $2$-transitively on an algebraic variety. Finally, in subsection~\ref{ss:high-transitive} we establish that a highly transitive group cannot be virtually solvable; see Corollary~\ref{cor:virt-solv}. 
In particular, such a group is of exponential growth provided it satisfies a Tits' type alternative; cf.~Question~2.

\section{Preliminaries from toric geometry}\label{sec:Preliminaries}
We start by recalling the standard notation and definitions of toric geometry. 

\subsection{Toric affine varieties} 
\label{sit:1.1} 
  Consider 
an algebraic torus $\mathbb{T}= (\mathbb{G}_m)^n$. 
Let $N$ be the lattice of one-parameter subgroups of $\mathbb{T}$, 
$N^\vee = {\rm Hom}\,(\mathbb{T},\mathbb{G}_m)$ the dual lattice of characters, and $\langle \cdot , \cdot \rangle\colon N \times N^\vee \to {\mathbb Z}$ the natural pairing; see, e.g.,~\cite{CLS:2011} or [6, Def.~4.2]. 
Let $\chi^m$ stand for the character of $\mathbb{T}$ which corresponds to  $m \in N^\vee$, so that $\chi^m\chi^{m'} = \chi^{m+m'}$. 
The group algebra
${\bf k}[N^\vee]= \oplus_{m\in N^\vee} {\bf k}\chi^m$ can be identified with the structure algebra $\mathcal{O}(\mathbb{T})$. 

Consider further the pair of dual ${\mathbb Q}$-vector spaces $N_{\mathbb Q}=N\otimes {\mathbb Q}$ and $N^\vee_{\mathbb Q}=N^\vee\otimes {\mathbb Q}$, 
a closed polyhedral cone $\Delta_{\mathbb Q}\subset N_{\mathbb Q}$ 
and its dual cone $\Delta^\vee_{{\mathbb Q}}\subset N^\vee_{\mathbb Q}$. 
By abuse of language, by the associated pair of \emph{lattice cones}  we mean the pair  $(\Delta,\Delta^\vee)$, where
$\Delta=N\cap \Delta_{\mathbb Q}$ and $\Delta^\vee=N^\vee\cap \Delta^\vee_{{\mathbb Q}}$, respectively. 
With any (closed) polyhedral lattice cone $\Delta\subset N$  one associates the normal toric affine variety 
\[X=\mathop{\rm Spec}\, \left(\bigoplus_{m\in \Delta^\vee} {\bf k}\chi^m\right),\] and any normal toric affine variety arises in this way.
The $\mathbb{T}$-action on $X$ is induced by the $\Delta^\vee$-grading on the structure algebra $\mathcal{O}(X)$.
 By Gordon's Lemma \cite[Prop.~1.2.17]{CLS:2011}, the cones $\Delta$ and $\Delta^\vee$ are both finitely generated monoids. 
The lattice vectors $(\rho_j)_{j=1,\ldots,k}$ on the extremal rays of $\Delta_{{\mathbb Q}}$, 
which are elements of the minimal system of generators 
 of $\Delta$, are called \emph{ray generators}. These are in one-to-one correspondence with the facets of the dual polyhedral cone $\Delta^\vee_{\mathbb Q}$.
The variety $X$ has no torus factor if and only if $\Delta_{\mathbb Q}$ 
is full dimensional, if and only if
$\Delta^\vee_{\mathbb Q}\subset N^\vee_{\mathbb Q}$ is pointed, that is, contains no affine line. 
See also~\cite{CLS:2011, Ful:1993, Oda:1988} for a detailed exposition.

\subsection{Root derivations and root subgroups}\label{def:roots} 
 The \emph{Demazure facets} 
\[S_j=\{e\in N^\vee\,|\,\langle\rho_j,e\rangle=-1, \quad\langle\rho_i,e\rangle\ge 0,\quad i=1,\ldots,k,\,\,\, i\neq j\},\quad j=1,\ldots,k\]  
are  lattice polyhedra in $N^\vee$.
Each of them is contained in a hyperplane of $N^\vee_{\mathbb Q}$ parallel to a facet of $\Delta^\vee_{{\mathbb Q}}$ and  situated outside the cone $\Delta^\vee_{{\mathbb Q}}$.
The lattice vectors $e\in\bigcup_{j=1}^k S_j$ are called \emph{Demazure  roots}. Any Demazure facet contains an infinite set of Demazure roots. 
To a Demazure  root  $e\in S_j$ one associates the \emph{root derivation} $\partial_{\rho_j, e}\in {\rm Der}\,(\mathcal{O}(X))$,
 which acts on the character $\chi^m$ via
\[\partial_{\rho_j, e} (\chi^m)=\langle \rho_j,m\rangle \chi^{m+e}.\] 
The kernel of $\partial_{\rho_j, e}$ is spanned by the characters $\chi^{m}$, where $m$ runs over the facet  of $\Delta^\vee$ defined by $\langle \rho_j,m\rangle=0$.

The root derivations are precisely the homogeneous locally nilpotent derivations of the graded algebra  $\mathcal{O}(X)= \bigoplus_{m\in \Delta^\vee} {\bf k}\chi^m$. Recall \cite[Thm.~2.4]{Lien:2010} that any homogeneous derivation of $\mathcal{O}(X)$ is proportional to one of the form $\partial_{\rho,e}$ for some $\rho\in N$ and  $e\in N^\vee$ acting via \[\partial_{\rho,e}(\chi^m)=\langle\rho, m\rangle \chi^{m+e}, \] where $e$ is called the \emph{degree} of $\partial_{\rho,e}$. 
One has \cite[Sect.~3]{Rom:2014} 
\begin{equation}\label{eq:Rom}
 [\partial_1,\partial_2]=\partial_{\rho, e_1+e_2}\quad\text{with}\quad \rho=d\rho_2-c\rho_1.
\end{equation} 
The  root subgroups $\exp(t\partial_{\rho_j, e})$ are precisely the $\mathbb{G}_a$-subgroups 
of $\mathop{\rm Aut} (X)$ normalized by the torus $\mathbb{T}$.  See, e.g.,~\cite{AKZ:2019, ALS:2021, Fre:2017, 
Lien:2010} for further details. 

\subsection{Cox rings and total coordinates}\label{sit:1.3} 
Let $X$ be a normal toric affine variety $X$ with no torus factor. The divisor class group ${\rm Cl}(X)$ is the abelian group  generated by the classes of the prime $\mathbb{T}$-invariant divisors $D_1,\ldots,D_k$ on $X$. These divisors are in one-to-one correspondence with the ray generators $(\rho_j)_{j=1,\ldots,k}$.
The \emph{Cox ring} of $X$ is the polynomial ring $\mathcal{O}(\mathbb{A}^k)={\bf k}[x_1,\ldots,x_k]$ on a distinguished set of variables called the \emph{total coordinates}. 
It is equipped with a ${\rm Cl}(X)$-grading defined by $\deg(x_i)=[D_i]$, $i=1,\ldots,k$. This grading corresponds to a diagonal action  on $\mathbb{A}^k=\mathop{\rm Spec}({\bf k}[x_1,\ldots,x_k])$ of the \emph{Cox quasitorus} 
$F_{\rm Cox}={\rm Hom}\,({\rm Cl}(X),\mathbb{G}_m)$. Recall that a \emph{quasitorus} is a direct product of an algebraic torus and a finite abelian group. One has \cite[Thm.~2.1.3.2]{ADHL:2015}
\[X\cong\mathop{\rm Spec}(\mathcal{O}(\mathbb{A}^k)^{F_{\rm Cox}})=\mathbb{A}^k/\!/F_{\rm Cox}.\]
See also ~\cite{ADHL:2015, AKZ:2019},~\cite[Ch.~5]{CLS:2011}.

\begin{lemma}\label{lem:lifting}
Let $e\in S_j$ be a Demazure root, and let $\hat e=(c_1,\ldots,c_k)\in {\mathbb Z}^k$, where $c_i=\langle\rho_i,e\rangle$. Then the following hold.
\begin{itemize}\item[{\rm (a)}] The integer lattice vector $\hat e$ is a Demazure root of $\mathbb{A}^k$ (viewed as a toric variety with the standard action of the $k$-torus) which belongs to the $j$th Demazure facet $\hat S_j$ of the first octant ${\mathbb Z}^k_{\ge 0}\subset {\mathbb Z}^k$.
\item[{\rm (b)}] Let $(\varepsilon_i)_{i=1,\ldots,k}$ be the ray generators of the lattice cone ${\mathbb Z}^k_{\ge 0}$. Then 
one has  
\begin{equation}\label{eq:elementary1} \hat \partial:=\partial_{\varepsilon_j,\hat e} = M_j\frac{\partial}{\partial x_j},\quad\text{where}
\quad M_j=x_1^{c_{1}}\cdots x_{j-1}^{c_{j-1}}x_{j+1}^{c_{j+1}}\cdots x_k^{c_{k}}\in{\bf k}[x_1,\ldots,x_{j-1},x_{j+1},\ldots,x_k].\end{equation}
The associated $\mathbb{G}_a$-subgroup consists of elementary transformations
\begin{equation}\label{eq:elementary2} 
\exp(t\hat \partial):(x_1,\ldots,x_k)\mapsto 
(x_1,\ldots, x_{j-1}, x_j+tM_j,x_{j+1},\ldots, x_k), \,\,\, t\in {\bf k}.
\end{equation}
This is a  subgroup of the tame automorphism group ${\rm Tame}\,(\mathbb{A}^k)$.
\item[{\rm (c)}] The Cox quasitorus $F_{\rm Cox}$ and the $\mathbb{G}_a$-subgroup $\exp(t\hat \partial)$ commute in $\mathop{\rm Aut} (\mathcal{O}(\mathbb{A}^k))$, and
\begin{equation}\label{eq:elementary4}
\exp(t\hat \partial)|_{\mathcal{O}(\mathbb{A}^k)^{F_{\rm Cox}}}=\exp(t\partial_{\rho_j,e}).
\end{equation}
\item[{\rm (d)}] Given a sequence $(\partial_1,\ldots,\partial_s)$ of root derivations  of $\mathcal{O}(X)$, where $\partial_i=\partial_{\rho_{j(i)},e_i}$ with a Demazure root $e_i\in S_{j(i)}$ of $X$, and the sequence  of the corresponding root derivations $\hat\partial_i=\partial_{\varepsilon_{j(i)},\hat e_i}$ of the Cox ring $\mathcal{O}(\mathbb{A}^k)={\bf k}[x_1,\ldots,x_k]$ with $\hat e_i\in \hat S_{j(i)}$, $i=1,\ldots,s$,
consider the Lie algebras $L$ and $\hat L$ generated, respectively, by $\partial_1,\ldots,\partial_s$ and $\hat\partial_1,\ldots,\hat\partial_s$. Then the correspondence $\partial_i\mapsto \hat\partial_i$, $i=1,\ldots,s$, induces an isomorphism of 
Lie algebras $L\cong\hat L$. 
\end{itemize}
\end{lemma}

\begin{proof}
Statement (a) is immediate; statements (b) and (c) follow easily from \cite[Lem.~4.20.b]{AKZ:2019}; see~\cite[(12)]{AKZ:2019} for~\eqref{eq:elementary1}. To show (d), consider the morphism $\pi\colon \mathbb{A}^k\to X=\mathbb{A}^k/\!/F_{\rm Cox}$. The induced  pullback homomorphism  $\pi^*\colon \mathcal{O}(X)\to \mathcal{O}(\mathbb{A}^k)$ is injective, and its image coincides with the algebra of invariants $\mathcal{O}(\mathbb{A}^k)^{F_{\rm Cox}}$. The induced homomorphism of the Lie algebras of vector fields $\pi^*\colon {\rm Vec}(X)\to {\rm Vec}(\mathbb{A}^k)$  is as well injective, and its image coincides with the Lie subalgebra of $F_{\rm Cox}$-invariant vector fields on $\mathbb{A}^k$ yielding an isomorphism ${\rm Vec}(X)\cong {\rm Vec}(\mathbb{A}^k)^{F_{\rm Cox}}$. Considering the derivations as vector fields, we have $\pi^*(\partial_i)=\hat\partial_i$, $i=1,\ldots,s$, and $\pi^*(L)=\hat L$. 
\end{proof}

Recall that a linear algebraic group is called \emph{unipotent} if it consists of unipotent matrices. In characteristic zero, any unipotent algebraic group is isomorphic to an affine space $\mathbb{A}^n$ as a variety. Any orbit of  a unipotent  algebraic group acting regularly on an affine variety is  closed and isomorphic to an affine space. 
In the sequel we need the following technical results.

\begin{proposition}\label{prop:pushforward} 
Given a collection of Demazure roots $\left(e_{j(i),i}\in S_{j(i)}\right)_{i=1,\ldots,s}$, let  
\[G=\langle U_i\,|\,i=1,\ldots,s\rangle\subset\mathop{\rm Aut} (X)\quad\text{where}\quad U_i=\exp(t\partial_{\rho_{j(i)},e_{j(i),i}}).\] 
Consider the root derivations  $\hat\partial_i=\hat\partial_{\varepsilon_{j(i)}, \hat e_{j(i),i}}$ and the root subgroups 
$\hat U_i=\exp(t\hat\partial_i)$ acting on $\mathbb{A}^k$, $i=1,\ldots,s$. Let
\[ \hat G=\langle \hat U_i\,|\,i=1,\ldots,s\rangle\subset\mathop{\rm Aut} (\mathbb{A}^k).\]
Then the following holds.
\begin{itemize}\item[{\rm (a)}]  If $\hat G$ contains a free subgroup ${\bf F}_m$ of rank $m\ge 2$ then  $G$ does.
\item[{\rm (b)}] If $\hat G$ is a unipotent algebraic group then $G$ is, and,  moreover, $G\cong\hat G$. 
\end{itemize}
\end{proposition}

\begin{proof}  (a) 
Since any subgroup $\hat U_i$, $i=1,\ldots,s$ commutes with the quasitorus  $F_{\rm Cox}$ in $\mathop{\rm Aut} (\mathbb{A}^k)$ one has
\[\hat G\subset {\rm Centr}_{\mathop{\rm Aut} (\mathbb{A}^k)}(F_{\rm Cox})\subset {\rm Norm} _{\mathop{\rm Aut} (\mathbb{A}^k)}(F_{\rm Cox}),\]
where ${\rm Centr}_{\mathop{\rm Aut} (\mathbb{A}^k)}(F_{\rm Cox})$ and ${\rm Norm} _{\mathop{\rm Aut} (\mathbb{A}^k)}(F_{\rm Cox})$ are the centralizer and the normalizer of $F_{\rm Cox}$ in $\mathop{\rm Aut} (\mathbb{A}^k)$, respectively.  
There is the exact sequence \cite[Thm.~5.1]{AG:2010}
\begin{equation}\label{eq:AG} 
1\to F_{\rm Cox}\to {\rm Norm} _{\mathop{\rm Aut} (\mathbb{A}^k)}(F_{\rm Cox})\stackrel{\tau}{\longrightarrow}\mathop{\rm Aut} (X)\to 1.
\end{equation}
Assume  $\hat G$ contains a free subgroup $\mathbf{F}_m$ of rank $m\ge 2$. We claim that  the restriction 
\[\tau|_{\mathbf{F}_m} \colon \mathbf{F}_m\to \mathbf{F}_m/(\mathbf{F}_m\cap F_{\rm Cox})\subset \mathop{\rm Aut} (X)\] is an isomorphism, that is, $\mathbf{F}_m\cap F_{\rm Cox}=1$. Indeed, the latter intersection is a normal abelian subgroup of the non-abelian free group $\mathbf{F}_m$, hence the trivial group. 

(b) Suppose $\hat G$ is a unipotent algebraic group. Then, once again,  the restriction 
\[\tau|_{\hat G} \colon \hat G\to \hat G/(\hat G\cap F_{\rm Cox})\subset \mathop{\rm Aut} (X)\] 
is an isomorphism, that is, $\hat G\cap F_{\rm Cox}=1$. Indeed, the unipotent linear algebraic group $\hat G$ has no torsion. Hence, $\hat G\cap F_{\rm Cox}$ is an algebraic subgroup of the quasitorus $F_{\rm Cox}$ with no torsion, therefore, the trivial group. 
\end{proof}

\begin{remarks} \footnote{This remarks is due to a referee.} In the proof we have used the fact that $\mathbf{F}_m$, $m > 1$, does not contain any normal abelian subgroup. This can be deduced as follows. Such an abelian subgroup $A$ is cyclic because any subgroup of $\mathbf{F}_m$ is free. It fixes two points on the boundary $\partial\mathbf{F}_m$, namely, the ends of the Caley graph of $A\cong\mathbb{Z}$. These two points form an invariant set of $\partial\mathbf{F}_m$ provided $A$ is normal. However, no
finite set is fixed by the $\mathbf{F}_m$-action on $\partial\mathbf{F}_m$. 

More generally, no nontrivial abelian subgroup of $\mathbf{F}_m$, $m > 1$, is subnormal (see subsection~\ref{ss:high-transitive} for the definition).  
Indeed, assume there is 
a descending series $\mathbf{F}_m\unrhd N_1\unrhd \ldots\unrhd N_s\unrhd A$, where $A\neq 1$ is abelian, hence  a free cyclic group. 
One may suppose that $N_s$ is a non-abelian free group of finite rank, and then the previous result applied to the pair $(N_s, A)$ gives a contradiction. 
\end{remarks} 

\section{Tits' type alternative for a pair of root subgroups}\label{sec:2-sbgrps}

In this section we  still deal with a toric affine variety $X$ over ${\bf k}$ with no torus factor, 
and freely use the notation from \ref{sit:1.1}--\ref{sit:1.3}.
We prove the following partial result; cf.\ Theorem~\ref{thm:0.4}. 

\begin{proposition}\label{prop:1.3} Consider the group $H=\langle U_1,\,U_2\rangle\subset\mathop{\rm Aut} (X)$ generated by the root subgroups $U_i=\exp(t\partial_i)$,  $i=1,2$, 
associated with two different ray generators, say, $\rho_1$ and $\rho_2$, respectively. Then either $H$ is a unipotent algebraic group, or 
$H$ contains a free subgroup of rank $2$. 
\end{proposition}

\begin{proof} 
Introducing the total coordinates $(x_1,\ldots,x_k)$, 
we let $U_1$ and $U_2$ act on $\mathbb{A}^k$ as 
$\mathbb{G}_a$-subgroups $\hat U_1$ and $\hat U_2$ of the tame automorphism group ${\rm Tame}\,(\mathbb{A}^k)$ commuting with the Cox quasitorus
$F_{\rm Cox}$, see Lemma \ref{lem:lifting}. 
We let $\hat H=\langle\hat U_1,\hat U_2\rangle$.  By Proposition \ref{prop:pushforward} it suffices to prove the above alternative for $\hat H$ instead of $H$.

Let in these coordinates $\hat e_i=(c_{ij})$ where $c_{ii}=-1$ and $c_{ij}\ge 0$ for $j\neq i$, $i\in\{1,2\}$.
One can write
\[\hat e_1=(-1,c,*,\ldots,*)\quad\text{and}\quad \hat e_2=(d,-1,*,\ldots,*),\quad\text{where}\quad c=\langle\rho_2,e_1\rangle,\quad d=\langle\rho_1,e_2\rangle,\] 
and the stars stand for nonnegative integers. The elements $\hat u_i\in \hat U_i$, $i=1,2$ can be written as 
\begin{equation}\label{eq:ui} \hat u_1=(x_1+sx_2^cN_1,x_{2},\ldots, x_k)\quad\text{and}\quad 
\hat u_2=(x_1,x_2+tx_1^dN_2,x_{3},\ldots, x_k),\end{equation} 
where $s,t\in{\bf k}$ and $N_1,N_2\in{\bf k}[x_3,\ldots,x_k]$ are nonzero monomials, cf.\ \eqref{eq:elementary1}--\eqref{eq:elementary2}. 

By \eqref{eq:Rom}, $\hat H$ is abelian (and  then $\hat H\cong \mathbb{G}_a\times\mathbb{G}_a$) if and only if $c=d=0$.  
 More generally, the following holds.

\medskip

\noindent {\bf Claim 1.} \emph{Assume $c>0$ and $d=0$. Then $\hat H=\langle \hat U_1,\hat U_2\rangle$ is a unipotent linear algebraic group.}

\medskip

\noindent \emph{Proof of Claim 1.} Under our assumptions, $\hat H$ is a closed subgroup of the unipotent linear algebraic group consisting of the triangular transformations
\[(x_1,\ldots,x_k)\mapsto (x_1 +F(x_2,N_2)N_1,x_2 +tN_2,x_3,\ldots,x_k),\]
where $t\in {\bf k}$ and $F$ runs over the linear space of homogeneous polynomials in two variables of degree $c$. So, $\hat H$ is a unipotent linear algebraic group.~\footnote{Alternatively, one can deduce the conclusion by using Proposition \ref{prop:2-cycle}.} \qed 

\smallskip
 
Suppose now that $c\ge 1$ and $d\ge 1$. In this case we show,
using ping-pong type arguments,
that $\hat H$ contains a free subgroup of rank two, see Claims 2--4. We analyze separately the cases $c,d\ge 2$, $c\ge 2$ and $d=1$, and $c=d=1$.
This analysis is close to the original Jung approach  in~\cite{Jun:1942}; cf.\ also~\cite[Lem.~4.1]{Kam:1979} and~\cite[5.31, p.~65]{Wr:1975}. 
Another reference in order is \cite{Lam:2001}, where the enhanced Tits alternative for the group $\mathop{\rm Aut} (\mathbb{A}_{\mathbb{C}}^2)$ was established playing the ping-pong on the Bass-Serre tree. On can apply this alternative to  the group $\mathop{\rm Aut} (\mathbb{A}_{K}^2)$, where $K$ is the rational function field ${\bf k}(x_3,\ldots,x_n)$.

\smallskip
 
Notice that by \eqref{eq:ui}, any $\hat h\in \hat H$ can be written as 
\begin{equation}\label{eq:h} \hat h=(p,q,x_3,\ldots,x_k)\quad\text{with}\quad p,q\in{\bf k}[x_1,\ldots,x_k]\setminus {\bf k}.\end{equation}

\medskip

\noindent {\bf Claim 2.} \emph{Assume $c,d\ge 2$. 
Then one has $\hat H=\hat U_1*\hat U_2\cong\mathbb{G}_a*\mathbb{G}_a$. 
Consequently,  any 
two non-unit elements $\hat u_i\in \hat U_i$, $i=1,2$, generate a free subgroup of rank two. }

\medskip

\noindent \emph{Proof of Claim 2.}
Fixing $\hat u_i\in \hat U_i$, $i=1,2$ as in \eqref{eq:ui} with nonzero $ s,t\in{\bf k}$, for $\hat h$ as in \eqref{eq:h} one has 
\[\hat u_1 \hat h=(p_1,q,x_3,\ldots,x_k)\quad\text{and}\quad \hat u_2 \hat h=(p,q_2,x_3,\ldots,x_k),\]
where by \eqref{eq:ui},
\begin{equation}\label{eq:formulas} p_1=p+sq^cN_1\quad\text{and}\quad q_2=q+tp^dN_2. \end{equation}
For $\deg(p)\le \deg(q)$ one gets
\begin{equation}\label{eq:u1} \deg(p_1)=c\deg(q)+\deg\,(N_1)>\deg(q), \end{equation}
and, similarly, for $\deg(p)\ge\deg(q)$ 
one deduces 
\begin{equation}\label{eq:u2} \deg(q_2)>\deg(p).\end{equation}

Consider a nontrivial reduced word $w$ in two letters, and let $\hat h=w(\hat u_1,\hat u_2)\in \hat H$, where $\hat u_1, \hat u_2\neq 1$. 
Using \eqref{eq:u1}--\eqref{eq:u2}
one concludes by recursion on the length of $w$
that $\deg(p)>\deg(q)$ if $w(\hat u_1,\hat u_2)$ 
starts on the left with $\hat u_1$, and $\deg(p)<\deg(q)$ if $w(\hat u_1,\hat u_2)$ starts with $\hat u_2$. Anyway,
$\deg(p)\neq \deg(q)$, hence $\hat h\neq 1$.  \qed

\medskip 

\noindent {\bf Claim 3.} \emph{Assume $c\ge 2$ and $d=1$. Then $\langle \hat u_1,\hat u_2\rangle$ is a free subgroup 
of rank two for any non-unit elements $\hat u_1\in \hat U_1,\,\hat u_2\in \hat U_2$.}

\medskip

\noindent \emph{Proof of Claim 3.}
The Jung-van der Kulk Theorem \cite{Jun:1942, vdK:1953} implies  the presentation
\[{\mathop{\rm Aut} }(\mathbb{A}^2)=A*_C J,\]
where $C=A\cap J,$ $A={\rm Aff}(\mathbb{A}^2)$ is the affine group, and $J$ is the de Jonqui\`eres subgroup of 
${\mathop{\rm Aut} }(\mathbb{A}^2)$
consisting of the transformations of the form 
\[(x_1,x_2)\mapsto (\alpha_1 x_1+\beta_1(x_2), \alpha_2x_2+\beta_2)\quad\text{with }\quad
\alpha_i\in {\bf k}^*,\,\, i=1,2,\,\,\, \beta_1\in {\bf k}[x_2],\,\,\,\beta_2\in {\bf k}; \]
see \cite{Dic:1983, Nag:1972, Wr:1975}, \cite[Thm.~2]{Kam:1975}, and \cite[Lem.~4.1]{Kam:1979}. 

Let $\hat u_1=\hat u_1(s)$, $\hat u_2=\hat u_2(t)$, and $N_1, \,N_2\in{\bf k}[x_3,\ldots,x_k]$ be as in \eqref{eq:ui}.
Pick a point $P_0=(x_3^0,\ldots,x_k^0)\in\mathbb{A}^{k-2}$ with only  nonzero coordinates, so that both $N_1(P_0)$ and $N_2(P_0)$ do not vanish.  
Letting
\begin{equation}\label{eq:coeff} 
b_1(s)=sN_1(P_0),\quad b_2(t)=tN_2(P_0),\quad \hat u^{(0)}_1=\hat u_1(s,P_0),\quad  \hat u^{(0)}_2=\hat u_2(t,P_0)
\end{equation}
one gets
\begin{equation}\label{eq:specialization} \langle \hat u^{(0)}_1,\,\hat u^{(0)}_2\rangle=
\langle (x_1+b_1(s)x^c_2, x_2), \,\,(x_1, x_2+b_2(t)x_1)\rangle\subset\mathop{\rm Aut} (\mathbb{A}^2),
\end{equation}
where $b_1(s),\,b_2(t)$ do not vanish for any $(s,t)\in (\mathbb{A}^1\setminus\{0\})^2$. 
Since $c>1$ and $d=1$, for any nonzero $m\in {\mathbb Z}$ one has $(\hat u_1^{(0)})^m\in J\setminus C$ 
and $(\hat u_2^{(0)})^m\in A\setminus C$ provided $(s,t)\in (\mathbb{A}^1\setminus\{0\})^2$ is fixed.  
Write $\hat h\in \langle \hat u^{(0)}_1,\hat u^{(0)}_2\rangle$ as \[\hat h=w(\hat u^{(0)}_1,\hat u^{(0)}_2),\] 
where $w$ is a reduced word in two letters.
Applying  the Jung-van der Kulk Theorem to $w(\hat u^{(0)}_1,\hat u^{(0)}_2)$ 
we conclude that $\hat h=1$ 
if and only if $w$ is the trivial word. 
Thus, one has 
$\langle \hat u^{(0)}_1,\hat u^{(0)}_2\rangle\cong \mathbf{F}_2$ for any fixed $(s,t)\in (\mathbb{A}^1\setminus\{0\})^2$. 
The specialization $(\hat u_1,\hat u_2)\mapsto (\hat u^{(0)}_1,\hat u^{(0)}_2)$ 
defines an isomorphism $\langle \hat u_1,\hat u_2\rangle \cong \mathbf{F}_2$.
\qed

\medskip

The next claim ends the proof of Proposition~\ref{prop:1.3}.

\medskip 

\noindent {\bf Claim 4.} \emph{Assume $c=d=1$. 
Then there exist $(\hat u_1,\hat u_2)\in \hat U_1\times \hat U_2$ such that the group $\langle \hat u_1,\hat u_2\rangle$ surjects onto $\mathop{\rm SL}_2({\mathbb Z})$ 
 and so, 
contains a free subgroup of rank two.}

\medskip

\noindent \emph{Proof of Claim 4.} We repeat the argument from the proof of Claim 3. Choosing in \eqref{eq:coeff} the value of $(s,t)$ such that $b_1(s)=b_2(t)=1$,  by \eqref{eq:specialization} we obtain
\[\langle \hat u_1^{(0)},\,\hat u_2^{(0)}\rangle=\langle (x_1+x_2, x_2), \,\,(x_1, x_2+x_1)\rangle= {\mathop{\rm SL}}_2({\mathbb Z}).\] 
This yields the desired surjection $\langle \hat u_1,\hat u_2\rangle\to\mathop{\rm SL}_2({\mathbb Z})$. 
It remains to recall \cite[3.1]{Wiki} 
that ${\mathop{\rm SL}}_2({\mathbb Z})$ is virtually free with 
$\langle (x_1+2x_2, x_2), \,\,(x_1, x_2+2x_1)\rangle\cong \mathbf{F}_2$.
\end{proof}

\begin{corollary}\label{cor:3cases} In the notation as before, 
the following conditions are equivalent:
\begin{itemize}
\item[(i)]
$\hat H=\langle \hat U_1,\hat U_2\rangle$ is a unipotent algebraic group;
\item[(ii)] letting $K={\bf k}[x_3,\ldots,x_k]$
the Lie algebra $\hat L={\rm Lie}(\hat\partial_1,\hat\partial_2)\subset {\rm Der}_K(K[x_1,x_2])$ 
generated by the root derivations (see~\eqref{eq:ui})
\[\hat\partial_1=x_2^cN_1\partial/\partial x_1,\quad \hat\partial_2=x_1^dN_2\partial/\partial x_2\]
is finite dimensional and nilpotent;
\item[(iii)]
\begin{equation}\label{eq:min} \min\{\langle \hat\rho_1,\hat e_2\rangle,\,
\langle\hat\rho_2,\hat e_1\rangle\}=\min\{c,d\}=0.
\end{equation}
\end{itemize}

These equivalences remain valid after taking off the hats.
\end{corollary}

\begin{proof}
The implication (i) $\Rightarrow$ (ii) is immediate; indeed, 
$\hat L={\rm Lie}(\hat H)$ provided (i) is fulfilled. 
The equivalence (i) $\Leftrightarrow$ (iii) is established in the course of the proof of Proposition~\ref{prop:1.3}.
Hence, it suffices to show (ii) $\Rightarrow$ (iii). Notice that the specialization 
$(x_1,\ldots,x_k)\mapsto (x_1,x_2,x_3^{(0)},\ldots,x_k^{(0)})$ yields a surjective homomorphism of Lie algebras
$ {\rm Der}_K(K[x_1,x_2])\to{\rm Der}_{\bf k}( {\bf k}[x_1,x_2])$.
Choosing a point $P_0=(x_3^{(0)},\ldots,x_k^{(0)})\in (\mathbb{A}^1\setminus\{0\})^{k-2}$ so that $N_1(P_0)$ and $N_2(P_0)$ 
do not vanish we may assume that $\hat\partial_1=y^c\partial/\partial x$, $\hat\partial_2=x^d\partial/\partial y$, and 
$\hat L={\rm Lie}(\hat\partial_1,\hat\partial_2)\subset  {\rm Der}_{\bf k}( {\bf k}[x,y])$ is finite dimensional and nilpotent. 
Suppose (iii) fails. If $c=d=1$ then $\hat L={\rm sl}_2({\bf k})$ is not nilpotent. If, say, $c\ge 1$ and $d>1$ then we have
\[{\rm ad}(\hat\partial_1)^{d+1}(\hat\partial_2)=-c(d+1)! y^e\partial/\partial x\in\hat L,\quad\text{where}\quad e=(d+1)c-1>c.\]
Replacing now $\hat\partial_1=y^c\partial/\partial x$ by $y^e\partial/\partial x$ and repeating the trick, 
we obtain a sequence of elements of $\hat L$ of unbounded degrees. Thus, in this case $\hat L$ has infinite dimension. 
In any case, (ii) fails, a contradiction. 

For the last assertion, see Lemma~\ref{lem:lifting}.d and Proposition~\ref{prop:pushforward}.b. 
\end{proof}

\section{Tits' type alternative for a sequence of root subgroups}\label{sec:TA}
Let as before $X$ be a  toric  affine variety with no torus factor, and let 
\[G=\langle U_1,...,U_s\rangle\] 
be the group generated by a finite set of root
subgroups $U_j=\exp(t\partial_j)\subset {\rm Aut}(X)$, $j=1,\ldots,s$, where $\partial_j$ are root derivations. 
According to Corollary~\ref{cor:3cases}, in the case that
$G$ does not contain any non-abelian free subgroup, for any $i\neq j$ either $U_i$ and $U_j$ 
belong to the same ray generator (and then commute), 
or they belong to two different ray generators $\rho$ and $\rho'$ 
and for the corresponding roots $e,e'$ one has $\min\{\langle \rho, e'\rangle, \langle\rho', e\rangle\}=0$.
In Proposition~\ref{thm:2} we establish that under these assumptions $G$ is a unipotent algebraic group. 
To be more precise, notice that 
the Lie algebra $L$ generated by the root derivations $\partial_j$, $j=1,\ldots,s$, might contain extra root derivations, 
cf.\ Example~\ref{ex:3coord}.  
Let $R_i$ be the set of Demazure roots $e_{ij}\in S_i$ of $X$ such that $\partial_{\rho_i, e_{ij}}\in L$. A priori, the cardinal ${\rm card}(R_i)$ could be infinite countable, and then the abelian subalgebra
\begin{equation}\label{eq:Li} L_i= {\rm Lie}(\partial_{\rho_i,e_{ij}} | e_{ij} \in R_i) \subset L\end{equation}
is infinite dimensional. 
 We may suppose that 
\[R_i\neq \emptyset\quad\forall i=1,\ldots,r\quad\text{and}\quad R_i=\emptyset\quad \forall i=r+1,\ldots,k.\] 
Let $R=\bigcup_{i=1}^r R_i$.
For $e\in R_i$ we let $U_e=\exp(t\partial_{\rho_i, e})$. In this section we prove the following proposition.

\begin{proposition}\label{thm:2}
Suppose that for all $e,e'\in R$ the group $\langle U_e, U_{e'}\rangle$ is unipotent. Then $G$ is a unipotent algebraic group.
\end{proposition}

The proof is done at the end of the section. The assumption of Proposition~\ref{thm:2} 
is equivalent to the fact that $\langle U_e, U_{e'}\rangle$ for any $e,e'\in R$ does not contain 
any free subgroup of rank two. The latter is equivalent to the fact that \eqref{eq:min} 
holds for any $e,e'\in R$, see Proposition~\ref{prop:1.3}  and Corollary~\ref{cor:3cases}.
Theorem~\ref{thm:0.4} from the introduction is an immediate consequence of 
Propositions~\ref{prop:1.3} and~\ref{thm:2}. 
In turn,  Proposition~\ref{thm:2} follows from Propositions~\ref{prop:ALS} and~\ref{prop:2-cycle}.

\subsection{Acyclicity and nilpotent Lie algebras}\label{ss:anla}
Before passing to the proof of Proposition~\ref{thm:2}, let us give an example. 

\begin{example}\label{ex:3coord}
Consider the group $G=\langle U_1, U_2, U_3, U_4\rangle$ acting on $\mathbb{A}^3=\mathop{\rm Spec} {\bf k}[x,y,z]$,  where
$U_i=\exp(t\partial_i)$, $i=1,\ldots,4$ with
\[\partial_1=yz\frac{\partial}{\partial x},\quad \partial_2=z\frac{\partial}{\partial y},\quad \partial_3
=z^2\frac{\partial}{\partial y},\quad \partial_4=\frac{\partial}{\partial z}.\]
We have \[\partial_1=\partial_{\rho_1,e_1}, \quad \partial_2=\partial_{\rho_2,e_2},\quad \partial_3=\partial_{\rho_2,e_3},\quad\partial_4=\partial_{\rho_3,e_4},\] where the ray generators $\rho_1, \rho_2, \rho_3$ are the vectors of the standard basis in $\mathbb{A}^3$, and
\[e_1=(-1,1,1),\quad e_2=(0,-1,1),\quad e_3=(0,-1,2),\quad e_4=(0,0,-1).\]
Any pair of these root derivations verify \eqref{eq:min}. They generate the Lie algebra
\[L={\rm span}\left(\frac{\partial}{\partial x},\,y\frac{\partial}{\partial x},\,
yz\frac{\partial}{\partial x},\,z\frac{\partial}{\partial x},\,z^2\frac{\partial}{\partial x},\,z^3\frac{\partial}{\partial x},\,
\frac{\partial}{\partial y},\,z\frac{\partial}{\partial y},\,z^2\frac{\partial}{\partial y},\,\frac{\partial}{\partial z}\right).\] 
Consider the abelian Lie subalgebras \[L_1={\rm span}\left(\frac{\partial}{\partial x},\,
y\frac{\partial}{\partial x},\,yz\frac{\partial}{\partial x},\,z\frac{\partial}{\partial x},\,
z^2\frac{\partial}{\partial x},\,z^3\frac{\partial}{\partial x}\right),\]
\[L_2={\rm span}\left( \frac{\partial}{\partial y},\,z\frac{\partial}{\partial y},\,
z^2\frac{\partial}{\partial y}\right),\quad\text{and}\quad L_3={\rm span}\left(\frac{\partial}{\partial z}\right).\]
We have \[L=L_1\oplus L_2\oplus L_3,\quad\text{where}\quad [L_1,L_i]\subset L_1,\,\,i=2,3, \quad [L_2,L_3]\subset L_2,\] 
and, furthermore, 
\[{\rm ad}(L_i)(L_i)=0,\,\,i=1,2,3,\qquad {\rm ad}(L_3)^4(L_1)=0,\qquad  {\rm ad}(L_2)^2(L_1)=0,\qquad  {\rm ad}(L_3)^3(L_2)=0.\]
For the lower central series $L^i=[L,L^{i-1}]$ 
 of $L$ we obtain $L^5=0$.
Thus, $L$ is nilpotent, and so, by Proposition~\ref{prop:2-cycle}, $G$ is a unipotent algebraic group. 
\end{example}

The proof of  Proposition~\ref{thm:2} is based on Proposition~\ref{prop:ALS}, 
which strengthens \cite[Thm.~5.1]{ALS:2021} in our particular context.
Let us recall the terminology of \cite{ALS:2021} and introduce the necessary notation.

\begin{definition}\label{sit:ALS} Consider a finite sequence  
of root derivations
\[\mathcal{D}=(D_1,\ldots,D_{t}, D_{t+1})\,\,\,\text{where}\,\,\, D_i=\partial_{\rho_{j(i)},\,e_{j(i),i}}\in L_{j(i)}
\,\,\,\text{with}\,\,\, e_{j(i),i}\in R_{j(i)},\,\,\,j(i)\in\{1,\ldots,r\}.\]
One says that $\mathcal{D}$ is a \emph{cycle} (more precisely, a \emph{$t$-cycle}) if $D_{t+1}=D_1$ and
\begin{equation}\label{eq:cycle} \langle\rho_{j(i+1)}, e_{j(i),i}\rangle >0\quad \forall i=1,\ldots,t. \end{equation}
 For instance, $(D_1,D_2,D_1)$ forms a $2$-cycle if and only if \eqref{eq:min} fails, that is, 
\[\langle\rho_{j(2)}, e_{j(1),1}\rangle >0\quad\text{and}\quad \langle\rho_{j(1)}, e_{j(2),2}\rangle >0.\]
We say that $\mathcal{D}$ is a \emph{pseudo-cycle} if \eqref{eq:cycle} holds and $j(t+1)=j(1)$, 
but not necessarily $e_{j(t+1), t+1}=e_{j(1),1}$; that is, 
$\rho_{j(t+1)}=\rho_{j(1)}$ but possibly $D_{t+1}\neq D_1$.
\end{definition}

In this subsection we mainly deal with the case where $G$ contains no non-abelian free subgroup, 
or, which is equivalent, $L$ contains no 2-cycle of root derivations. We need the next technical lemma. 

\begin{lemma}\label{lem:ideal} The following conditions are equivalent:
\begin{itemize}\item[{\rm (i)}] 
$L$ contains no $2$-cycle of root derivations;
\item[{\rm (ii)}] 
$L$ contains no $2$-pseudo-cycle of root derivations; 
\item[{\rm (iii)}] 
for any pair of indices $i,j\in \{1,\ldots,r\}$ such that $i\neq j$, at least one of 
the abelian Lie subalgebras $L_i, L_j$ from~\eqref{eq:Li} 
is an ideal of the Lie algebra ${\rm Lie}\, (L_i, L_j)$;
\item[{\rm (iv)}] the Lie algebra $\rm{Lie}(\partial_{\rho_i,e},\,\partial_{\rho_j,e'})$ 
is finite dimensional and nilpotent for any pair of indices $i,j\in \{1,\ldots,r\}$ and any pair of roots $e\in R_i$, $e'\in R_j$. 
\end{itemize}
\end{lemma}

\begin{proof} (i) $\Leftrightarrow$ (ii). 
Assume (i) holds. Then we have
\begin{equation}\label{eq:assumption}
\min\{\langle \rho_i,e'\rangle,\,\langle \rho_j,e\rangle\}=
0\qquad\forall e\in R_i,\quad \forall e'\in R_j\quad\text{with}\quad 1\le i\neq j\le r.
\end{equation}
Condition (ii) is clearly fulfilled if $L_i$ and $L_j$ commute. 
 Otherwise, up to interchanging $i$ and $j$, there exists $e_i\in R_i$ such that
$\langle \rho_j,e_i\rangle=c>0$, see~\eqref{eq:Rom}. By virtue of
 \eqref{eq:assumption} one has
\begin{equation}\label{eq:zero-coord} \langle \rho_i,e'\rangle=0\quad\forall e'\in R_j.\end{equation}  
It follows that $L$ has no 2-pseudo-cycle, that is, (ii) holds. 
The converse implication is immediate. 

 (ii) $\Leftrightarrow$ (iii). Assume (ii) holds. Then~\eqref{eq:assumption} is fulfilled. 
 As before, (iii) is evidently true if $L_i$ and $L_j$ commute. Suppose this is not the case, 
 and let $\langle \rho_j,e_i\rangle=c>0$ for some $e_i\in R_i$.
 From  \eqref{eq:Rom} and \eqref{eq:zero-coord}  one deduces that $e_i+e'\in R_i$ for any $e'\in R_j$, and
\begin{equation}\label{eq:commutant} [\partial_{\rho_j,e'},\,\partial_{\rho_i,e_i}]=c\partial_{\rho_i,e_i+e'}\in 
L_i\,\,\,\forall e'\in R_j,
\end{equation} 
that is,
\begin{equation}\label{eq:inclusion} 0\neq [L_j, L_i]\subset L_i.\end{equation} 
Thus, (iii) is fulfilled. To show the converse, notice that $[L_j, L_i]\subset L_i$ for $i\neq j$ 
implies~\eqref{eq:commutant} for any $e_i\in R_i,\,e'\in R_j$ with $c=\langle\rho_j,e_i\rangle$, 
and also implies $\langle\rho_i,e'\rangle=0$. 
Thus,~\eqref{eq:assumption} holds, and so, one has the implication (iii) $\Rightarrow$ (ii). 

 The equivalence (iv) $\Leftrightarrow$ (i) holds by Corollary~\ref{cor:3cases}. 

\end{proof} 

\begin{definition}\label{def:graph} To any Lie algebra $L$ as before 
we associate a directed graph $\Gamma_r=\Gamma_r(L)$ on $r$ vertices $\{L_i\}_{i=1,\ldots,r}$, 
where a directed edge $[L_j\rightarrow L_i]$ joins the vertices $L_i$ and $L_j$ 
if and only if  $\langle \rho_j,e_i\rangle>0$ for some $e_i\in R_i$. 
\end{definition}

Thus, there is no edge joining the vertices $L_i$ and $L_j$ of $\Gamma_r$ 
if and only if $[L_i, L_j]=0$, 
that is, the Lie algebra ${\rm Lie}\, (L_i, L_j)$ is abelian. Furthermore, $\Gamma_r$ 
has no bidirected edge if and only if $L$ has no 2-pseudo-cycle of root derivations. 
For instance, this holds for the following graph $\Gamma_3=\Gamma_3(L)$ 
associated with the Lie algebra $L$ from Example~\ref{ex:3coord}:
\begin{equation}
\label{diag:gamma}
\Gamma_3:\quad\vcenter{
\xymatrix@R=1em{
L_3\ar[dr]\ar[rr]&&L_2\ar[dl]
\\
&L_1&
} }
\end{equation}

\begin{lemma}\label{lem:2-cycle} The following are equivalent:
\begin{itemize}\item[{\rm (i)}] $L$ contains no pseudo-cycle  of root derivations;
\item[{\rm (ii)}] $L$ contains no cycle  of root derivations; 
\item[{\rm (iii)}] $L$ contains no $2$-cycle  of root derivations. 
\end{itemize}
\end{lemma}

\begin{proof} 
It suffices to prove (iii)$\Rightarrow$(i), the two other implications being immediate. 

Assume $L$ contains no 2-cycle, and then also no 2-pseudo-cycle of root derivations, see Lemma~\ref{lem:ideal}. 
Suppose to the contrary that  $L$ has a pseudo-cycle  of root derivations $\mathcal{D}=\{D_1,\ldots,D_N, D_{N+1}\}$ with $N\ge 3$. 
Then $\Gamma_r$ has the oriented cycle
\[L_{\rho_{j(1)}}\rightarrow L_{\rho_{j(N)}}\rightarrow\ldots\rightarrow L_{\rho_{j(2)}}\rightarrow L_{\rho_{j(1)}}.\] 
The sequence $\rho_{j(1)},\ldots,\rho_{j(N)}$ of the corresponding ray generators can eventually  contain repetitions. 
However, it is possible to subtract a subsequence 
$\rho_{j(1)},\ldots,\rho_{j(t)}$  without repetition, where $3\le t\le N$, such that $\rho_{j(t+1)}=\rho_{j(1)}$. 
Then $\mathcal{D}'=\{D_1,\ldots,D_t, D_{t+1}\}$ is again a pseudo-cycle, and the cycle 
\[L_{\rho_{j(1)}}\rightarrow L_{\rho_{j(t)}}\rightarrow\ldots\rightarrow L_{\rho_{j(2)}}\rightarrow L_{\rho_{j(1)}}\] 
has no self-intersection.
To any $e\in R$ we associate the integer vector of length $t$,
\[v(e)=(\langle\rho_{j(1)},e\rangle,\ldots, \langle\rho_{j(t)},e\rangle)\in{\mathbb Z}^t.\]
One has
\begin{equation}\label{eq:matrix}
\begin{aligned} v(e_{j(1),1})&=(-1,\bullet,*,\ldots,*,*)\\
v(e_{j(2),2})&=(0,-1,\bullet,*,\ldots,*)\\
v(e_{j(3),3})&=(*,0,-1,\bullet,\ldots,*)\\
\,&\,\,\, \vdots\\
v(e_{j(t-1),t-1})&=(*,*,\ldots,0,-1,\bullet)\\
v(e_{j(t),t})&=(\bullet,*,*,\ldots,0,-1)
\end{aligned}
\end{equation}
The stars in~\eqref{eq:matrix} stand for nonnegative integers, the bullets stand for positive integers,
and the zeros on the lower subdiagonal
 are due to \eqref{eq:cycle} and  \eqref{eq:assumption}. In fact, \eqref{eq:cycle} 
 and  \eqref{eq:assumption} imply
 \begin{equation}\label{eq:zero-product} 
 \langle\rho_{j(i)},\,e\rangle=0\quad\forall e\in R_{j(i+1)}.
 \end{equation}
From  \eqref{eq:Rom}, \eqref{eq:cycle} and \eqref{eq:zero-product}  one deduces 
\[\langle\rho_{j(i+2)},e'_{j(i),i}\rangle>0\quad\text{where}\quad e_{j(i),i}':=e_{j(i),i}+e_{j(i+1),i+1}\in R_{j(i)},\,\,\,i=1,\ldots,t-2.\]
Then \eqref{eq:assumption} gives \[\langle\rho_{j(i)},e\rangle=0
\quad\forall e\in R_{j(i+2)}.\]
This means that the second lower subdiagonal in \eqref{eq:matrix} consists of zeros. 
In the same fashion one can show that the third lower subdiagonal in \eqref{eq:matrix} consists of zeros. Finally, we arrive by recursion 
to the conclusion that the matrix in \eqref{eq:matrix} 
is upper triangular. Moreover, one has
\[\langle\rho_{j(t+1)},e\rangle=\langle\rho_{j(1)},e\rangle=0\quad\forall e\in R_{j(t)}.\] 
The latter contradicts \eqref{eq:cycle} for $i=t$ and $e=e_{j(t),t}$.
\end{proof}

The following statement strengthens Theorem~5.1 in \cite{ALS:2021} in application to our (simpler) setup. 
For the convenience of the reader we provide a complete proof, which exploits Lemma~\ref{lem:2-cycle}.

\begin{proposition}\label{prop:ALS}
Assume $L$ contains no 2-cycle of root derivations. 
Then the associated graph $\Gamma_r$ is acyclic, that is, does not contain any oriented cycle, and the Lie algebra $L$ is finite-dimensional and nilpotent.
\end{proposition}

\begin{proof}
We freely use the notation from the proof of Lemma \ref{lem:2-cycle}. Consider the one-dimensional 
Lie subalgebras $l_{\rho_i,e_i}$ of $L_i$ generated by the root derivations, where
\[l_{\rho_i,e_i}={\rm span}(\partial_{\rho_i,e_i}) ={\bf k}\partial_{\rho_i,e_i}\quad\text{with}\quad e_i\in R_i.\] 
Since $L$ has no 2-cycle
then \eqref{eq:assumption} holds. Hence, for $i\neq j$ there is the alternative:
\begin{equation}\label{eq:alternative} \text{either}\quad [l_{\rho_i,e_i}, l_{\rho_j,e_j}]=0,\quad\text{or}\quad  
[l_{\rho_i,e_i}, l_{\rho_j,e_j}]\in \{l_{\rho_i,e_i+e_j},\, l_{\rho_j,e_i+e_j}\}. 
\end{equation}
Due to \eqref{eq:Rom} one has
\begin{equation}\label{eq:equiv}
[l_{\rho_i,e_i}, l_{\rho_j,e_j}]=l_{\rho_i,e_i+e_j} \quad\text{if and only if}\quad \langle \rho_j,e_i\rangle>0.
\end{equation}
In the latter case $\Gamma_r$ contains the directed edge $[L_j\rightarrow L_i]$. 
It is clear that \[L_i=\bigoplus_{e\in R_i} l_{\rho_i,e}\quad\text{and}\quad L=\bigoplus_{i=1}^r L_i.\]
Therefore, one has 
\begin{equation}\label{eq:dim} \dim(L)=\sum_{i=1}^r \dim(L_i)=\sum_{i=1}^r {\rm card}\,(R_i)={\rm card}\,(R).\end{equation}
Let us show that under our assumptions $\Gamma_r$ is acyclic, that is, does not contain any oriented cycle. 
Indeed, given an oriented cycle  in $\Gamma_r$,
\[L_{j(1)}\to L_{j(2)}\to\ldots\to L_{j(t)}\to L_{j(t+1)}=L_{j(1)},\] one can find a sequence of roots  
$e_{j(i),i}\in R_{j(i)}$ such that, with the usual convention $\rho_{j(t+1)}=\rho_{j(1)}$,  one has
\[\langle\rho_{j(i+1)}, e_{j(i),i}\rangle>0, \quad i=1,\ldots,t.\] 
Thus, $\mathcal{D}=(D_i=\partial_{\rho_{j(i)},\,e_{j(i),i}})_{i=1,\ldots,t}$ is a pseudo-cycle of root derivations in $L$. 
By Lemma~\ref{lem:2-cycle}, 
the latter contradicts our assumption on absence of 2-cycles in $L$.  

A vertex $L_i$ is called a \emph{sink} if either $L_i$ is isolated in $\Gamma_r$, or
all the incident edges of $\Gamma_r$ at $L_i$ have the incoming direction, that is, $L_i$ does not emit any edge. The vertex $L_i$  of $\Gamma_r$ is
a sink  if and only if $L_i$ is an ideal of the Lie algebra $L$. 

The end vertex of any maximal oriented path in $\Gamma_r$ is a sink. Since $\Gamma_r$ is acyclic it has at least one sink.
Moreover, any connected component of $\Gamma_r$ contains a sink. 

We can choose a new enumeration of the vertices of $\Gamma_r$  taking for $L_1$ a vertex which is a sink of $\Gamma_r$. Deleting $L_1$ from 
$\Gamma_r$ along with its incident edges yields a directed graph $\Gamma_{r-1}$. 
The corresponding Lie subalgebra of $L$ still has no pseudo-cycle of root derivations. 
Hence, $\Gamma_{r-1}$ has at least one sink. 
We choose a sink of $\Gamma_{r-1}$ to be $L_2$, etc. 
By construction, with this new enumeration one has (cf.\ Example \ref{ex:3coord})
\begin{equation}\label{eq:commutator-relations} 
\begin{aligned} 
\,[L_i,L_1] & \subset L_1, \quad i=2,\ldots,r,\\
[L_i,L_2] & \subset L_2, \quad i=3,\ldots,r,\\
\,&\,\,\, \vdots\\
\,\,\,[L_{r},L_{r-1}] & \subset L_{r-1},\\
\,\,\,\,\,\,[L_{r},L_{r}] &= 0\,.
\end{aligned} 
\end{equation}

To show that $L$ is of finite dimension, we use the enumeration of the subalgebras $L_j\subset L$ satisfying \eqref{eq:commutator-relations}. 
We establish that the total coordinates of the vectors in $R$ are uniformly bounded above, and so, $R$ is finite.  Due to \eqref{eq:dim} this yields the result.

At the beginning of Section~\ref{sec:TA} we defined $L$ as the Lie algebra generated by the finite set of root derivations $\partial_i$, $i=1,\ldots,s$. 
Given a  ray generator $\rho_j$, consider all the root derivations $\partial_i$ among $\partial_1,\ldots,\partial_s$ which belong to $\rho_j$, 
and let $R_j^{(0)}\subset R_j$ be the set of their roots. 
It follows from \eqref{eq:alternative} and \eqref{eq:commutator-relations} that $R_r=R_r^{(0)}$, and so, $R_r$ is finite. 
Furthermore, by \eqref{eq:zero-coord}, \eqref{eq:inclusion}, and \eqref{eq:commutator-relations} for any $e\in R_r$  one has
\[ \langle\rho_i,\,e\rangle=0\quad\forall i=1,\ldots,r-1\quad\text{and}\quad \langle\rho_r,\,e\rangle=-1.
\]
Again by \eqref{eq:alternative} and \eqref{eq:commutator-relations}, any root $e\in R_{r-1}$ is of the form  
\begin{equation}\label{eq:decomp} e=e^{(0)}_{r-1}+e_{r,1}+\ldots+e_{r,m}\quad\text{with}\quad e_{r-1}^{(0)}\in R_{r-1}^{(0)}\quad\text{and}\quad 
e_{r,i}\in R_r,\,\,\,i=1,\ldots,m,
\end{equation}
where the lattice vectors $e_{r,i}\in R_r$ are not necessarily distinct. We claim that \[0\le m\le \langle \rho_{r},\,e^{(0)}_{r-1}\rangle.\] 
Indeed, for the $r$th total coordinate of the lattice vectors in  \eqref{eq:decomp}  we have
\[\langle \rho_{r},\,e_{r,i}\rangle=-1, \,\,\,i=1,\ldots,m, \quad\text{and}\quad \langle \rho_{r},\,e\rangle=\langle \rho_{r},\,e^{(0)}_{r-1}\rangle-m\ge 0.\] 
Since both $R_{r-1}^{(0)}$ and $R_r$ are finite, we conclude that $R_{r-1}$ is as well. 

Suppose by induction that the $R_i$ are finite for $i=t,\ldots,r$, where $2\le t\le r-1$. By \eqref{eq:alternative} and \eqref{eq:commutator-relations},  any root $e\in R_{t-1}$ is of the form  
\begin{equation}\label{eq:double-summation} e=e^{(0)}_{t-1}+\sum_{i=t}^r\sum_{j=1}^{m_i} e_{i,j} \quad\text{with}\quad e_{t-1}^{(0)}\in R_{t-1}^{(0)}\quad\text{and}\quad e_{i,j}\in R_i,
\end{equation} 
with possible repetitions of the summands. 
 Likewise in \eqref{eq:matrix}, due to the chosen enumeration  we obtain for the first $i$ total coordinates of the vector $e_{i,j}\in R_i$,
\begin{equation}\label{eq:tot-coord}
\langle \rho_l,  e_{i,j}\rangle=0\quad \forall l=1,\ldots,i-1\quad\text{and}\quad \langle \rho_i,  e_{i,j}\rangle=-1,\,\,\,j=1,\ldots,m_i.
\end{equation}
Letting $l=t$ yields 
\begin{equation}\label{eq:triangular-matrix}
\langle \rho_t,  e_{i,j}\rangle=\begin{cases} \,\,\,0, &  i=t+1,\ldots,r\\ -1, & i=t\,.\end{cases}
\end{equation}
From \eqref{eq:double-summation}--\eqref{eq:triangular-matrix} we deduce 
\[\langle \rho_l, e\rangle=0\quad \forall l<t-1,\quad \langle \rho_{t-1}, e\rangle=-1,\quad\text{and}\quad
\langle \rho_t, e\rangle=\langle \rho_t, e^{(0)}_{t-1}\rangle-m_t\ge 0. \] 
Therefore, one has 
\[m_t\le\langle \rho_t, e^{(0)}_{t-1}\rangle=:\tilde m_t.\]
To find a uniform bound for the $(t+1)$st total coordinate $\langle \rho_{t+1}, e\rangle$ of $e$ we let 
\[\tilde m_{t+1}=\max_{e^{(0)}_{t-1}\in R^{(0)}_{t-1}}\{\langle \rho_{t+1}, e^{(0)}_{t-1}\rangle\}+\max_{m_t\le \tilde m_t} 
\Big\{\sum_{j=1}^{m_t} \langle \rho_{t+1}, e_{t,j}\rangle\,|\,  (e_{t,1},\ldots,e_{t,m_t})\in R_t^{m_t}\Big\}<+\infty.\] 
Arguing as before we obtain $m_{t+1}\le \tilde m_{t+1}$. Continuing in this way we arrive at the conclusion that all 
the total coordinates of the vectors $e$ from $R_{t-1}$  are uniformly bounded above, and so, $R_{t-1}$ is finite. 
This gives the induction step. 
Thus, $R$ is finite, and the dimension $\dim (L)={\rm card}\,(R)$ is finite too, see \eqref{eq:dim}.

Let us show finally that $L$ is nilpotent. 
Indeed, let $1\le i< j\le r$. Using the relations similar to \eqref{eq:tot-coord} and the fact that $R$ is finite, for $N\gg 1$ one deduces
\begin{equation}\label{eq:negative} \langle \rho_j, e+e_1+\ldots+e_N\rangle=\langle \rho_j, e\rangle-N\le -2\quad\text{whenever}\quad e\in R_i\quad\text{and}\quad e_1,\ldots,e_N\in R_j.
\end{equation}
Letting $l={\bf k} \partial_{\rho_i,e}\subset L_i$ and $l_k={\bf k} \partial_{\rho_j,e_k}\subset L_j$, $k=1,\ldots,N$ and using \eqref{eq:equiv} we obtain by \eqref{eq:negative}
\[
[l_1,[l_2,[\ldots,[l_N,l]\ldots ]=0\quad\text{whenever}\quad l_k\subset L_j, \,\,\,k=1,\ldots,N,\quad\text{and}\quad l\subset L_i,\,\,\,i\le j.\]
 For $N\gg 1$ the latter vanishing reads
\begin{equation}\label{eq:adjoint} 
{\rm ad}(L_j)^{N}(L_i)=0 \quad \forall j\ge i, \quad i,j\in \{1,\ldots,r\}.
\end{equation}
Taking into account \eqref{eq:commutator-relations}, from \eqref{eq:adjoint} we deduce
\[{\rm ad}(L)^{Nr}(L)=0,\] which means that $L$ is nilpotent.
\end{proof}

\subsection{From nilpotent Lie algebras to unipotent groups}\label{ss:laug}

It is well known, see, e.g., \cite[Ch.~XVI, Thm.~4.2]{Hoch:1981}, that over a field of characteristic zero,  
the Lie functor realizes the equivalence between the categories of unipotent algebraic groups 
and of nilpotent Lie algebras. In our particular case,  this correspondence can be made quite explicit. 

\begin{proposition}\label{prop:2-cycle} Let $X$ be a toric affine variety over $\bf k$ with no torus factor, let 
$G=\langle U_1,\ldots,U_s\rangle\subset {\rm Aut}(X)$ be a subgroup generated by the root subgroups $U_i=\exp(t\partial_i)$, where the $\partial_i$ are locally nilpotent derivations of the structure algebra $\mathcal{O}(X)$ associated with Demazure roots, let $L$ be the Lie algebra generated by $\partial_1,\ldots,\partial_s$, and let $\Gamma(L)=\Gamma_r(L)$ be the associated directed graph,  see Definition~\ref{def:graph}. Then the following are equivalent:
\begin{itemize}
\item[(i)]\label{i} $L$ has no $2$-cycle of root derivations; 
\item[(ii)]\label{ii} the graph $\Gamma(L)$ has no oriented cycle, in particular, no bidirected edge;
 \item[(iii)]\label{iii} $L$ is finite-dimensional and nilpotent;
  \item[(iv)]\label{iv} $G$ is a unipotent algebraic group acting regularly on $X$.
  \end{itemize}
   In the latter case one has $L={\rm Lie}\,(G)$.
\end{proposition}

\begin{proof} The implications (i) $\Rightarrow$ (ii)$\&$(iii) follow from Proposition~\ref{prop:ALS}. Condition (iii) implies (iv) of Lemma~\ref{lem:ideal}, and so, by virtue of this lemma, implies (i). Therefore, there is the equivalence (i) $\Leftrightarrow$ (iii). By Definition~\ref{def:graph}, $\Gamma(L)$ has no bidirected edge if and only if $L$ has no 2-pseudo-cycle of root derivations. By virtue of Lemma~\ref{lem:ideal} this is equivalent to (i). Thus, one has (i) $\Leftrightarrow$ (ii) $\Leftrightarrow$ (iii).

 If (iv) holds, then $G$ is a nilpotent group, and so, it contains no non-abelian free subgroup. This implies (i) due to Proposition~\ref{prop:1.3} and Corollary~\ref{cor:3cases}. Hence, we have the implications (iv) $\Rightarrow$ (i) $\Leftrightarrow$ (iii).  The converse implication  (iii) $\Rightarrow$ (iv) is proven in Lemmas~\ref{lem:4-01}--\ref{lem:4} below. 
 \end{proof}

\noindent {\bf Convention.}
We assume in the sequel that the Lie algebra $L$ is  finite-dimensional and nilpotent, and so, (i)--(iii) hold. We use the enumeration of the subalgebras $L_i\subset L$, $i=1,\ldots,r$ introduced in the proof of Proposition~\ref{prop:ALS}, so that \eqref{eq:commutator-relations} and \eqref{eq:adjoint} hold. 

\smallskip

Let us recollect the notation. Recall that  $k$ stands for the number of total coordinates on our toric affine variety $X$; this is at the same time the number of extremal rays 
$\rho_i$ of the cone 
$\Delta_{\mathbb Q}$, see subsections~\ref{sit:1.1} and~\ref{sit:1.3}.
 Due to~\eqref{eq:AG}, any automorphism of $X$ lifts to an automorphism of the Cox ring $\mathcal{O}(\mathbb{A}^k)={\bf k}[x_1,\ldots,x_k]$ of 
 $X$, and to an automorphism of the spectrum 
$\mathbb{A}^k$ of the Cox ring.   The algebra $\mathcal{O}(X)$ coincides with the algebra of the $F_{\rm Cox}$-invariants 
of the Cox ring $\mathcal{O}(\mathbb{A}^k)$. By Lemma~\ref{lem:lifting}.b any root derivation $\partial_i$, $i=1,\ldots,s$, lifts to a root derivation 
$\hat\partial_i$ of ${\bf k}[x_1,\ldots,x_k]$ of form~\eqref{eq:elementary1}, and any root subgroup $U_i=\exp(t\partial_i)$ 
lifts to a root subgroup 
$\hat U_i=\exp(t\hat \partial_i)$ consisting of elementary 
transformations~\eqref{eq:elementary2} and
centralized by the Cox quasitorus $F_{\rm Cox}$. The derivation $\partial_i$ is the restriction of $\hat\partial_i$ to $\mathcal{O}(\mathbb{A}^k)^{F_{\rm Cox}}\cong\mathcal{O}(X)$, $i=1,\ldots,s$. 
This yields an isomorphism 
$L\cong \hat L$, where $\hat L$ is the Lie algebra  
generated by $\hat\partial_1,\ldots,\hat\partial_s$, see Lemma~\ref{lem:lifting}.d. 

Assuming that  $\hat G=\langle \hat U_1,\ldots,\hat U_s\rangle$ is a unipotent linear algebraic group, one has an isomorphism $G\cong\hat G$, and so, $G$ is a unipotent linear algebraic group too, see Proposition~\ref{prop:pushforward}.b. 
Thus, it suffices to show that $\hat G$ 
is a unipotent linear algebraic group provided $\hat L$ is a finite-dimensional nilpotent Lie algebra. 

Let $\hat L_i$ be the image of $L_i$ under the isomorphism $L\cong\hat L$ from Lemma~\ref{lem:lifting}.d. 
It is easily seen that $L$ satisfies (i)--(iii) if and only if $\hat L$ does. 
 
Recall that the automorphisms of $\mathbb{A}^k$ of the form 
\begin{equation}\label{eq:triangular0} (x_1,\ldots,x_k)\mapsto 
(x_1+f_1,\ldots,x_k+f_k),\,\,\text{where}\,\, f_i\in  {\bf k}[x_{i+1},\ldots,x_k],\,\,\, i=1,\ldots,k,
\end{equation} 
are called \emph{unitriangular}. These automorphisms form the \emph{unitriangular subgroup} of the group ${\rm Aut}(\mathbb{A}^k)$, 
see~\cite[Ch.~3]{Fre:2017}. In Lemmas~\ref{lem:4-01} and~\ref{lem:4-02} we present $\hat G=\langle \hat U_1,\ldots,\hat U_s\rangle$ as a subgroup of the unitriangular group with ${\rm Lie}(\hat G)=\hat L$. 

\begin{lemma}\label{lem:4-01}
Any  $\partial\in L$ is a locally nilpotent derivation of $\mathcal{O}(X)$.
\end{lemma}

\begin{proof}
Due to \eqref{eq:elementary1},  in the total 
coordinates $(x_1,\ldots,x_k)$ 
any derivation $\hat\partial\in \hat{L}_i$, $i\in\{1,\ldots,r\}$, acts  on $\mathbb{A}^k$ via
\begin{equation}\label{eq:deriv} \hat\partial=p\partial/\partial x_i\quad\text{where}\quad 
p\in{\bf k}[x_1,\ldots,x_{i-1},x_{i+1},\ldots,x_k].
\end{equation}
 The monomials $M_j$ in \eqref{eq:elementary1}, taken up to proportionality, of all possible polynomials $p$ in \eqref{eq:deriv} 
are in one-to-one correspondence with the roots in the subset $R_j$ of $R$. Since $R_j$ is finite, 
it follows that 
\begin{equation}\label{eq:bound} 
\max_{\hat\partial\in \hat L_j}\,\{\deg(p)\}\le d_j\quad\text{for some}\quad d_j\in {\mathbb N}.
\end{equation}
Due to our choice of enumeration, the subalgebras $\hat L_i\subset \hat L$ satisfy \eqref{eq:commutator-relations}. Hence,
the total coordinates of the roots in $R_i$ 
form a triangular-like matrix, cf.~\eqref{eq:tot-coord}, that is, $p\in {\bf k}[x_{i+1},\ldots,x_k]$ for any $p$
 in \eqref{eq:deriv}.
Since $\hat\partial^2(x_i)=0$, $i=1,\ldots,k$, the $\mathbb{G}_a$-subgroup $\exp(t\hat\partial)$ of 
$\mathop{\rm Aut} (\mathbb{A}^k)$ generated by $\hat\partial$ from \eqref{eq:deriv} acts on $\mathbb{A}^k$ 
via the unitriangular (elementary) transformations
\begin{equation}\label{eq:elementary3} \exp(t\hat\partial)\colon (x_1,\ldots,x_k)\mapsto 
(x_1,\ldots,x_{i-1},x_i+tp(x_{i+1},\ldots,x_k),x_{i+1},\ldots,x_k),\,\,\,t\in{\bf k},\end{equation}
cf.~\eqref{eq:elementary2}.
More generally,
any derivation $\hat\partial\in \hat L$ is triangular of the form
\begin{equation}\label{eq:triangular-derivation} \hat\partial= \sum_{i=1}^r\delta_i=
\sum_{i=1}^r p_i\partial/\partial x_i,\quad\text{where}\quad \delta_i\in \hat L_i\quad\text{and}\quad p_i\in {\bf k}[x_{i+1},\ldots,x_k],\,\,i=1,\ldots,r.\end{equation}
According to \cite[Prop.~3.29]{Fre:2017}, $\hat\partial$ is locally nilpotent on $\mathcal{O}(\mathbb{A}^k)$, and then also $\partial$ 
is locally nilpotent on $\mathcal{O}(X)=\mathcal{O}(\mathbb{A}^k)^{F_{\rm Cox}}$. 
\end{proof}

Notice that for $r< k$ the variables $x_{r+1},\ldots,x_k$ belong to the kernel of any derivation $\hat\partial\in \hat L$. 

\smallskip

\noindent {\bf Classical formulas.} In what follows we use the classical Baker-Campbell-Hausdorff (BCH) and Zassenhaus formulas. 
Let us recall these formulas following \cite{BF:2012, LSS:2019, Man:2012, WGJ:2019}. 
Let $A$ be an associative algebra with unit over $\bf k$. In the formal power series algebra $A[[t]]$ the function $\exp$ is 
well defined for any series without constant term, and $\log$ is well defined for any series with the constant term equal $1$. 
For $a,b\in A$ consider the 
Lie subalgebra ${\rm Lie}(a,b)$ of $A$, 
that is, the smallest subspace of $A$ which contains $a$ and $b$ and is stable under commutators. The BCH formula
expresses $c(t):=\log(\exp(ta)\exp(tb)) \in {\rm Lie}(a,b)[[t]]$. Plugging in formally $t=1$, the formula reads:
\begin{equation}\label{eq:BCH} c:=c(1)= a+b+\frac{1}{2}[a,b]+\frac{1}{12}[a,[a,b]]+\frac{1}{12}[b,[b,a]]+\ldots,\end{equation}
where each term is a rational multiple of the Lie monomial in $a$ and $b$  of degree $m$ obtained from the universal 
Lie monomial $\alpha [z_1,[z_2,[\ldots,z_m]\ldots]$, $\alpha\in\mathbb{Q}$, after substitution of either $a$ or $b$ instead of every
$z_i$.

 Likewise, for given $a_1,\ldots, a_\nu\in A$,  the multivariate Zassenhaus formula reads:
\begin{equation}\label{eq:Zass} 
\exp(a_1+\ldots+a_\nu) =\exp(a_1)\cdots\exp(a_\nu)\prod_{m=2}^{\infty} \exp(\psi_m(a_1,\ldots,a_\nu)),
\end{equation} 
where $\psi_m(a_1,\ldots,a_\nu)$  is a homogeneous Lie polynomial in $a_1,\ldots,a_\nu$ 
of degree $m$. 

In our setup, it occurs that  the both formulas contain just a finite number of terms, and, respectively, factors.
 Indeed, 
letting $L= {\rm Lie}(a,b)$ in the former case and $L= {\rm Lie}(a_1,\ldots,a_\nu)$ in the latter case,
 suppose that $L$ is nilpotent of nilpotency class $n$, 
 that is, $L^{n+1}={\rm ad}(L)^n(L)=0$, where $n\in\mathbb{N}$ is minimal with this property. 
 Then the homogeneous Lie polynomials of degree $m>n$ in the both formulas vanish, hence the sum in~\eqref{eq:BCH} 
 and the product in~\eqref{eq:Zass} are finite.  
\smallskip

\begin{lemma}\label{lem:4-02} In the total coordinates, the automorphisms
$\exp(\hat L)=\{\exp(\hat\partial)\,|\,\hat\partial\in \hat L\}$ form a group 
of unitriangular automorphisms of~$\mathbb{A}^k$.
The map $\exp:\hat L\to \exp(\hat L)$ is well defined and bijective.
\end{lemma}

\begin{proof} By virtue of~\eqref{eq:elementary3} and~\eqref{eq:triangular-derivation}, any automorphism 
$\exp(\hat\partial)\in\exp(\hat L)\subset {\rm Aut}(\mathbb{A}^k)$ is unitriangular of the form
\begin{equation}\label{eq:triangular1} \exp(\hat\partial)\colon 
(x_1,\ldots,x_k)\mapsto (x_1+f_1,\ldots,x_r+f_r,x_{r+1},\ldots,x_k)\,\,\text{with}\,\, f_i\in  {\bf k}[x_{i+1},\ldots,x_k].
\end{equation} 
Any unitriangular  automorphism $\alpha\in\mathop{\rm Aut} (\mathbb{A}^k)$ can be written 
as the exponential $\alpha=\exp(c)$ of the triangular  derivation
\[c=\log(\alpha)=\log({\rm id}+( \alpha - {\rm id} ))\in\mathop{\rm Der}(\mathcal{O}(\mathbb{A}^k)),
\] 
 see \cite[Prop.~3.30]{Fre:2017} and its proof. 
Consider a pair $(a,b)$ of triangular derivations of $\mathcal{O}(\mathbb{A}^k)$ from $\hat L$. 
The product $\exp(a)\exp(b)$ of the corresponding unitriangular automorphisms is again unitriangular. In more detail,  $\exp(a)\exp(b)=\exp(c)$ with a triangular derivation 
\[c=\log(\exp(a)\exp(b)) \in\mathop{\rm Der}(\mathcal{O}(\mathbb{A}^k)),\] see \cite[Cor.~3.31]{Fre:2017}. The latter derivation 
can be expressed via the
BCH formula~\eqref{eq:BCH}  truncated on level $n+1$, where $n$ is the nilpotency class of $\hat L$. It follows that $c\in \hat L$. 
Thus, $\exp(\hat L)$ is a subgroup of the group of unitriangular automorphisms of $\mathbb{A}^k$. Since $\log$ and $\exp$ are
 mutually inverse, the map $\exp:\hat L\to \exp(\hat L)$ is a bijection.  
 \end{proof}

\begin{lemma}\label{lem:4-03} The degrees of $\exp(\hat \partial)$ are uniformly bounded for $\hat \partial \in \hat L$. 
\end{lemma}

\begin{proof}
Write $\hat \partial=a_1+\ldots+a_\nu\in \hat L$, where $a_i\in \hat L_i$. We can express $\exp(\hat \partial)=\exp(a_1+\ldots+a_\nu)$ 
via the  Zassenhaus formula
\eqref{eq:Zass}, where the product is truncated  on level $n+1$ with $n$ being the nilpotency class of $\hat L$. 

Consider the increasing chain of ideals of $\hat L$,
\[\mathcal{L}_1\subset\mathcal{L}_2\subset\ldots\subset\mathcal{L}_r=\hat L,
\quad\text{where}\quad \mathcal{L}_{\nu}:=\hat L_1\oplus\ldots\oplus \hat L_\nu,\]
see \eqref{eq:commutator-relations}. Notice that $\hat \partial=a_1+\ldots+a_\nu\in \mathcal{L}_{\nu}$.
Let us show that in \eqref{eq:Zass} one has $\psi_m(a_1,\ldots,a_\nu)\in \mathcal{L}_{\nu-1}$ for any $\nu\in\{2,\ldots,r\}$, $m\ge 2$. 
Indeed, it suffices to check this for $\psi_m$ which is a
Lie monomial in $a_1,\ldots,a_\nu$ of degree $m$. In this case, our claim follows from the fact that the abelian 
Lie algebras $\hat L_1,\ldots,\hat L_r$ verify \eqref{eq:commutator-relations}. 

Now we proceed by induction on $\nu=1,\ldots,r$. The assertion of the lemma holds for $\mathcal{L}_1$ 
due to~\eqref{eq:bound} and~\eqref{eq:elementary3}. Suppose it holds for some $\mathcal{L}_{\nu-1}$. Take \[\hat \partial=\sum_{i=1}^\nu a_i=
\sum_{i=1}^\nu p_i\partial/\partial x_i\in\mathcal{L}_\nu,\quad\text{where}\quad a_i\in \hat L_i.\] Using \eqref{eq:Zass}  
and the preceding observation, we can write
\begin{equation}\label{eq:decomposition} \exp(\hat \partial)=g_1\exp(a_\nu)g_2=
g_1\exp(p_\nu\partial/\partial x_\nu)g_2,\quad\text{where}\quad g_1,\,g_2\in \exp(\mathcal{L}_{\nu-1}).
\end{equation} 
Since our assertion holds for $\exp(\mathcal{L}_{\nu-1})$ by the induction hypothesis and for $\exp(\hat L_\nu)$ by 
\eqref{eq:bound}--\eqref{eq:elementary3}, the degrees of all the automorphisms in \eqref{eq:decomposition} 
are uniformly bounded above. Therefore, the assertion holds for $\exp(\mathcal{L}_{\nu})$ as well. This concludes the induction.  
\end{proof}

\begin{lemma}\label{lem:4}
Assume $L$ is a finite dimensional nilpotent Lie algebra. Then $G$ is a unipotent algebraic group acting regularly on $X$, and $L={\rm Lie(G)}$.
\end{lemma}

\begin{proof}
Due to Lemma~\ref{lem:4-03}, the span $F$ of $\exp(\hat L)$ is a finite-dimensional subspace of the vector space 
${\rm End}(\mathbb{A}^k)$. 
One can take for the coordinates in $\hat L$ and $F$ the coefficients of the polynomials  $p_1,\ldots,p_r$ and $f_1,\ldots,f_r$  in
\eqref{eq:triangular-derivation} and \eqref{eq:triangular1}, respectively.  The map 
\[\hat L\ni \hat \partial\mapsto \exp(\hat \partial)\in \exp(\hat L), \qquad (p_1,\ldots,p_r)\mapsto (f_1,\ldots,f_r),\] 
defines a morphism of algebraic varieties $\hat L\to F$. The image $\exp(\hat L)$ is an irreducible constructible subset of $F$. 
Since $\exp(\hat L)$ is a connected group, this 
 is a locally closed smooth subvariety of $F$. 
By Zariski's Main Theorem, the bijective morphism $\exp: \hat L\to\exp(\hat L)$ of smooth varieties is an isomorphism. 
Since $\hat L$ is an affine variety, $\exp(\hat L)$ is too. 
Using~\eqref{eq:triangular0}, it is easily seen that $\exp(\hat L)$ is an affine algebraic group which acts regularly on $\mathbb{A}^k$. 
The exponential of a nilpotent matrix is unipotent. Therefore, the group $\exp(\hat L)$ is unipotent since it consists of unipotent elements. 
 
Recall 
that $\hat G=\langle \hat U_1,\ldots,\hat U_s\rangle$, where $\hat U_j=\exp(t\hat \partial_j)$ with $\hat \partial_j\in \hat L$. 
So, $\hat U_j\subset\exp(\hat L)$ for any $j=1,\ldots,s$. 
It follows that $\hat G\subset \exp(\hat L)$. In fact, $\hat G=\exp(\hat L)$ because the Lie subalgebras ${\rm Lie}(\hat U_j)={\bf k}\hat \partial_j\subset \hat L$, $j=1,\ldots,s$, 
generate $\hat L$. 
By the preceding discussion we deduce ${\rm Lie}(\hat G)=\hat L$. It follows that ${\rm Lie}(G)=L$.
\end{proof}

\medskip

\noindent {\bf Proof of Proposition~\ref{thm:2}.} 
Due to Corollary \ref{cor:3cases}, under the assumption of Proposition~\ref{thm:2}, 
\eqref{eq:assumption} holds for any $e\in R_i,\,e'\in R_j$. 
Then $L$ has no $2$-cycle of root derivations. 
Now the assertion follows from Propositions~\ref{prop:ALS} and~\ref{prop:2-cycle}.
\hfill\qed

\section{Transitive actions}\label{sec:transitive}
\subsection{Doubly transitive groups acting on  toric affine varieties}\label{ss:2-transitive}

In this section we  apply the Tits' type alternative to answer Question \ref{ques:1} 
under the assumption of double transitivity of the group in question. 
We start with the following simple combinatorial lemma.

\begin{lemma}\label{lem:double-transitivity} 
The group which acts effectively and doubly transitively on a set of cardinality at least three has trivial center.
\end{lemma}

\begin{proof} 
Let a group $G$ act effectively and doubly transitively on a set $X$, where ${\rm card}(X)\ge 3$. 
Then the following hold.

\smallskip
\begin{enumerate}
\item[(a)]\label{claim:a} Any non-trivial normal subgroup $H$ of $G$  is transitive on $X$;
\item[(b)]\label{claim:b} the stabilizer $G_x$ of a point $x\in X$ 
acts transitively on $X\setminus\{x\}$, and so, $x$ is a unique fixed point of $G_x$. 
\end{enumerate}

\smallskip

\noindent To show (a) it suffices to notice that $G$ permutes the $H$-orbits on $X$. 
Statement (b) is immediate. 

\smallskip

Assume now that the center $Z$ of $G$ is nontrivial. Then by (a), $Z$ is transitive on $X$. 
On the other hand, since $Z$ commutes with $G_x$, it fixes the unique fixed point $x$ of $G_x$, see (b). 
This gives a contradiction. 
\end{proof}

The next statement follows immediately. 

\begin{proposition}\label{lem:unipotent-transitivity} No nilpotent group acts doubly transitively on a set $X$ with ${\rm card}(X)>2$. In particular,
a unipotent linear algebraic group cannot act $2$-transitively on an algebraic variety.
\end{proposition}

\medskip

\begin{remark}
Alternatively, the second statement can be deduced from the classification 
of doubly transitive groups of homeomorphisms and doubly transitive Lie groups, 
see~\cite{Kram:2003, Tit:1952, Tit:1956}, or by using \cite[Prop.~17.4 and Cor.~17.5]{Hum:1975}. 
\end{remark}

Now we can deduce Corollary~\ref{thm:0.3} from the Introduction.

\smallskip

\noindent {\bf Proof of Corollary~\ref{thm:0.3}.} 
The assertion follows immediately from Theorem~\ref{thm:0.4} and Proposition~\ref{lem:unipotent-transitivity}.
\qed


 \subsection{High transitivity of a subnormal subgroup}\label{ss:high-transitive}
 As we indicated in the Introduction, the highly transitive actions of the groups generated 
 by one-parameter unipotent subgroups were a starting point of this research. 
 We add below some results concerning a combinatorial aspect of high transitivity, 
 which could be useful in order to attack Conjecture~\ref{conj:general}.

 \begin{definition} Let $G$ be a group. We say that $G$ is \emph{highly transitive} if $G$ 
 admits an effective action on a set $X$ which is $m$-transitive for any $m\in {\mathbb N}$. 
 \end{definition}

\emph{Attention:}  one can find in the literature another definition of high transitivity, 
which does not require effectiveness. 

The non-abelian free groups provide examples of highly transitive groups \cite{Cam:1987, MD:1977}. 
 Recall that a subgroup
$N$ of a group $G$ is called \emph{subnormal} if there exists a descending normal series 
\begin{equation}\label{eq:normal-series} G \unrhd N_1 \unrhd N_2 \unrhd \ldots \unrhd N_k=N\,.
\end{equation}

\begin{proposition}\label{lem:subnormal} Assume that a group $G$ acts effectively 
and highly transitively on an infinite set $X$. Then any nontrivial subnormal subgroup $N$ 
of $G$ is also highly transitive on $X$. In particular, $N$ cannot be virtually solvable.
\end{proposition}
 
 The proof is done at the end of the section; it is based on the following lemma.  
In turn, the proof of the lemma imitates the one of \cite[Cor.~7.2A]{DM:1996}. 
For the sake of completeness, we provide the argument. 

\begin{lemma} \label{th-1}
Assume that a group $G$ acts effectively and highly transitively on a set $X$. 
Let $H$ be a non-trivial normal subgroup of $G$. Then $H$ acts on $X$
highly transitively.
\end{lemma}

\begin{proof} 
For any $m$-tuple $\alpha=\{x_1,\ldots,x_m\}$ of pairwise distinct points in $X$ we consider 
the stabilizers 
$$
G_{\alpha}=G_{x_1}\cap\ldots\cap G_{x_m} \quad
\text{and} \quad  H_{\alpha}=H_{x_1}\cap\ldots\cap H_{x_m}.
$$
Then $H_{\alpha}$ is a normal subgroup in $G_{\alpha}$. 

We have to show that for any positive integer $m$ and for any $m$-tuple $\alpha$ the group $H_{\alpha}$ acts transitively on 
$X\setminus\{x_1,\ldots,x_m\}$. 

By assumption, $G_\alpha$ acts  highly transitively on $X\setminus\{x_1,\ldots,x_m\}$. 
By Claim~(a) from the proof of Lemma~\ref{lem:double-transitivity},
either $H_\alpha$ is transitive on $X\setminus\{x_1,\ldots,x_m\}$, or $H_\alpha=\{e\}$.

Assuming the latter,  take the minimal $m$ with this property, where $m\ge 1$ 
by Claim~(a) from the proof of Lemma~\ref{lem:double-transitivity}. 
Let $\beta=\{x_1,\ldots,x_{m-1}\}$. By assumption,
the stabilizer $H_{\beta}$  is transitive on $X\setminus\{x_1,\ldots,x_{m-1}\}$ (for $m=1$ 
we have $H_{\beta}=H$, and the latter follows by Claim~(a) from the proof of 
Lemma~\ref{lem:double-transitivity}).  Moreover, $H_{\beta}$ is simply transitive  on 
$X\setminus\{x_1,\ldots,x_{m-1}\}$, and so, we can identify the set
 $X\setminus\{x_1,\ldots,x_m\}$  with $H_{\beta}\setminus\{e\}$ via the bijection
 $$
X\setminus\{x_1,\ldots,x_m\} \ni  y \ \mapsto \ h\in H_{\beta}\setminus\{e\}, \quad \text{where} \quad y=hx_m. 
 $$
Under this identification, the action of $G_\alpha$ on $X\setminus\{x_1,\ldots,x_m\}$ 
goes to the action of $G_\alpha$ on $H_{\beta}\setminus\{e\}$
by conjugation, due to the relation
$$
ghx_m=ghg^{-1}gx_m=ghg^{-1}x_m \quad \forall g\in G_\alpha, \quad \forall h\in H_{\beta}\setminus\{e\}. 
$$
The action by conjugation sends a pair $(h,h^{-1})$ to a pair of the same type.  
Since $H_{\beta}$ is infinite, it follows that the action of $G_x\subset \mathop{\rm Aut} (N)$ 
on $H_{\beta}\setminus\{e\}$ cannot be $2$-transitive, unless $H_{\beta}$ is a group of exponent two. 
  
Suppose finally that $H_{\beta}$ is a group of exponent two. It is well known 
that $H_{\beta}$ is a power of ${\mathbb Z}/2{\mathbb Z}$, or, in other words, 
the additive group of a vector space $V$ over the field $\mathbb{F}_2$ 
with two elements. However, the action of $\mathop{\rm Aut} (H_{\beta})={\rm GL}(V)$ 
is not $3$-transitive on $H_{\beta}\setminus\{e\}=V\setminus\{0\}$ contrary to our assumption, 
because it preserves the linear (in)dependence. This contradiction completes the proof.
 \end{proof}

\begin{remark}\label{rems} 
 Notice that the affine group $G={\rm Aff}\, (V)$ of the vector space $V=\mathbb{A}^n_{\mathbb{F}_2}$, $n\ge 3$, 
 acts $3$-transitively on $V$, while the normal subgroup of translations acts just simply transitively on $V$, 
 contrary to \cite[Exercise 2.1.16]{DM:1996}.
 \end{remark}

\begin{corollary}\label{cor:virt-solv}
A virtually solvable group cannot be highly transitive.
\end{corollary}

\begin{proof}
Any virtually solvable group $G$ contains a normal solvable subgroup $H$ of finite index.
In turn, $H$ contains a nontrivial normal abelian subgroup $A$, and $A$ contains a nontrivial
cyclic subgroup $C$. Clearly, a cyclic subgroup cannot be highly transitive. 
\end{proof} 

\begin{remark} By Gromov's Theorem~\cite{Gro:1981}, a finitely generated group has polynomial growth if and only if it is virtually nilpotent. 
Hence, the conclusion of Corollary~\ref{cor:virt-solv} holds for any finitely generated group of polynomial growth. 
 \end{remark}

\noindent {\bf Proof of Proposition~\ref{lem:subnormal}.}
 The first assertion of \ref{lem:subnormal} follows from Lemma~\ref{th-1} 
by recursion on the length of the normal series~\eqref{eq:normal-series}, 
and the second follows from Corollary~\ref{cor:virt-solv}.
\qed
  
\bigskip

\noindent{\bf Funding}

\medskip

\noindent The work of the first author was supported by the grant RSF-19-11-00172.

\bigskip

\noindent {\bf  Acknowledgements}

\medskip

\noindent The second author is grateful to Fran\c cois Dahmani, 
Rostislav Grigorchuk, Alexei Kanel-Belov, Boris Kunyavski, 
St\'ephane Lamy, Bertrand Remy, and  Christian Urech for useful conversations, letter exchange, and references. 
Our thanks are due as well to a referee for fruitful remarks and suggestions, which allowed to improve essentially 
the presentation and to strengthen one of the main results.
 

\end{document}